\documentclass[twoside,11pt]{article}
\usepackage{amsmath,amsthm}
\usepackage[preprint]{jmlr2e}
\usepackage{microtype,booktabs,mathtools,mathrsfs}
\usepackage{algorithm,algorithmic}
\usepackage{titletoc}
\usepackage[capitalize,noabbrev]{cleveref}
\numberwithin{theorem}{section}
\theoremstyle{plain}
\newtheorem{prop}[theorem]{Proposition}
\newtheorem{cor}[theorem]{Corollary}
\newtheorem{lem}[theorem]{Lemma}
\theoremstyle{remark}

\crefname{prop}{Proposition}{Propositions}
\Crefname{prop}{Proposition}{Propositions}
\crefname{cor}{Corollary}{Corollaries}
\Crefname{cor}{Corollary}{Corollaries}
\crefname{lem}{Lemma}{Lemmas}
\Crefname{lem}{Lemma}{Lemmas}
\crefname{rem}{Remark}{Remarks}
\Crefname{rem}{Remark}{Remarks}

\newcommand{\D}{\mathscr{D}}
\newcommand{\Dc}{\mathscr{D}_{\mathrm c}}
\newcommand{\EE}{\mathbb{E}}
\newcommand{\PP}{\mathbb{P}}
\newcommand{\R}{\mathbb{R}}
\newcommand{\Z}{\mathbb{Z}}
\newcommand{\T}{\mathbb{T}}
\newcommand{\F}{\mathcal{F}}
\newcommand{\Hc}{\mathcal{H}}
\newcommand{\M}{\mathcal{M}}
\newcommand{\Pol}{\mathcal{P}}

\ShortHeadings{Minimax Rates for imperfect Physics-Informed Regression}{Doum\`eche}
\firstpageno{1}
\begin{document}
\title{Minimax Rates for imperfect Physics-Informed Regression}
\author{\name Nathan Doum\`eche }
\maketitle

\begin{abstract}
Physics-informed machine learning incorporates physical knowledge into data-driven models,
often through a partial differential equation (PDE) that the target is assumed to satisfy.
However, how much this knowledge improves statistical performance remains unclear,
especially when the physical model is imperfect.
Since satisfying a PDE implies being differentiable, a natural benchmark for
physics-informed regression is the best achievable performance when smoothness
is the only prior information about the target function.
We show that there is a characteristic threshold for the modeling error
above which the physical prior cannot improve the convergence rate.
The form of this threshold is common to every linear PDE and certain nonlinear PDEs.
Below it, the gain depends on the equation type.
For classes of elliptic and hyperbolic PDEs with constant coefficients,
we establish convergence rates that improve continuously as the modeling error decreases,
and prove their minimax optimality up to logarithmic factors.
With exact physics, the rates correspond to learning in dimension $d-1$ rather than $d$,
reflecting the freedom left by unknown boundary or initial data.
We also give conditions for computing the estimator with near-linear complexity
and show how to tune its hyperparameters without knowing the modeling error.
\end{abstract}

\section{Introduction}
\paragraph{Physics-informed regression.} Physics-informed machine learning (PIML) combines observations with a physical model, 
typically a partial differential equation (PDE) $\D(f)=g$ satisfied on a domain $\Omega$. 
The physical prior is incorporated by penalizing the residual $\|\D(f)-g\|^2_{L^2(\Omega)}$ in the empirical risk function \citep{raissi2019PINN,karniadakis2021piml}.
In practice, such PDE models are often approximation of more complex physical processes, so that the unknown target function is not an exact solution to the PDE \citep{kennedy2001bayesian,tuo2015efficient, sangalli2021spatial, zhang2025sobolev}.
Moreover, satisfying a PDE implies differentiability. 
Indeed, if the PDE operator $\D$ involves derivatives of order $r$ (in which case it is said to be of order $r$), then a classical requirement for the PDE residual $\D(f)-g$ to be square-integrable is that the target function has weak derivatives of order $s\geq r$ in $L^2(\Omega)$.
Thus, the statistical properties of physics-informed regression must be compared with those of nonparametric regression over the class of functions with weak derivatives of order $s$, called the Sobolev space $H^s(\Omega)$.
We ask: \emph{when does a bound $\delta$ on the PDE residual allow faster estimation from $n$ samples than a bound $\rho$ on the Sobolev norm of order $s$ alone?}

For uniform design and Gaussian noise, with fixed $\rho$ and noise level, the minimax rate over Sobolev functions of order $s>d/2$ in dimension $d$ is $n^{-2s/(2s+d)}$ \citep{gyorfi2002distribution,caponnetto2007optimal}.
A natural guess is a bound of the form
\begin{equation}
\text{risk}\leq C\big(h(\delta)\,n^{-2s/(2s+d)}+n^{-2s/(2s+d-1)}\big),
\label{eq:naive}
\end{equation}
where $C$ does not depend on $n$ or $\delta$, $h(\delta)\to0$ as $\delta\to0$, and $n^{-2s/(2s+d-1)}$ is the Sobolev rate in dimension $d-1$, the minimax rate of exact solutions of the elliptic equations studied here on domains with a flat face, and of the hyperbolic ones up to a logarithmic factor.
Under such a bound, a small fixed residual would help at every sample size; for dependent data and elliptic operators of order $s$, the first term of a reported upper bound takes this form, subject to additional regularization and sample-size assumptions \citep{scampicchio2026physics}.
We prove that no such bound holds when $s>r$, where $r$ is the order of the PDE: what matters is the residual relative to a threshold that decreases with $n$.

\paragraph{Contributions.}
\begin{itemize}
\item \textbf{A universal benchmark} (\cref{thm:universal}).
For every linear PDE of order $r$, and a class of nonlinear PDEs, above a residual threshold of order $n^{-(s-r)/(2s+d)}$ the minimax risk remains of Sobolev order.
\item \textbf{Elliptic rates} (\cref{thm:elliptic,prop:ellexact}).
For constant coefficients and $d<2r$, a physics-informed kernel estimator attains this benchmark, up to the exact-solution rate in dimension $d-1$; matching lower bounds on domains with a flat face give the minimax rate without logarithmic factors.
\item \textbf{Hyperbolic rates} (\cref{thm:hyperbolic}).
For constant coefficients, $r\in\{1,2\}$ and $s>\max(r,d-r)$, the minimax rate is, up to logarithmic factors and for fixed radius and noise level, $n^{-2s/(2s+d-1)}+\min\{n^{-2s/(2s+d)},(\delta/n)^{2s/(3s+d-r)}\}$, slower than the benchmark when $r<d$.
\item \textbf{Effective dimension} (\cref{thm:count}).
We bound the effective dimension of the physics-informed kernel on bounded Lipschitz domains, without boundary values; the bounds distinguish elliptic and hyperbolic equations.
\item \textbf{Adaptation and computation} (\cref{thm:adaptation,thm:algo}).
Validation removes $\delta$ from the tuning under sub-Gaussian noise, at the price of a constant factor and an additive term of order $\log\log n/n$; given exact input matrices, the solver iterations cost $O(n\log^2n)$ under stated smoothness conditions.
\end{itemize}

\paragraph{Related work.}
Analyses of physics-informed neural networks and kernel methods give consistency and generalization bounds \citep{shin2020convergence,mishra2022generalization,deryck2024numerical,doumeche2025convergence,bonapellissier2026piks}.
Recent works quantify how physics speeds up learning through the effective dimension of physics-informed kernels \citep{doumeche2024physicsinformed,brogatmotte2026fast}, the dimension of the constrained parameter set \citep{koshizuka2025understanding}, or local complexities \citep{marcondes2025complexity,scampicchio2026physics}.
They give estimator upper bounds under geometric or complexity assumptions, with explicit examples on the torus, in dimension one, for gradient constraints, or for elliptic PDEs of order $s$; \citet{brogatmotte2026fast} leave minimax optimality and misspecified constraints open.
Minimax theory exists for neighboring problems: solving a PDE from noisy samples of its source term \citep{lu2022machine,zhao2026posterior}, recovering PDE coefficients \citep{nickl2019convergence}, estimating the state of exactly known dynamics \citep{li2024physicsencoded,konen2025data}, and learning with derivative observations \citep{dai2025nonparametric}.
In PDE-penalized spatial regression with known Dirichlet values, rates depend on the misfit, and exact physics can yield the parametric rate \citep{arnone2022consistency}.
Without boundary values, our exact-solution spaces are infinite-dimensional, with the Sobolev rate in dimension $d-1$, much as invariance under a one-parameter group removes one dimension \citep{tahmasebi2023exact}.
Our lower bounds hold for every estimator, including physics-informed neural networks; our upper bounds concern the physics-informed kernel estimator \citep{doumeche2024physicsinformed,doumeche2025pikl,doumeche2026fast}, analyzed by random-design ridge regression \citep{mourtada2022an,bach2024learning}.

\section{Physics-informed machine learning setting}
\label{sec:setting}
\paragraph{Regression model.}
We observe $n$ independent copies $(X_i,Y_i)$ of $(X,Y)$, where $Y=f^\star(X)+\varepsilon$, and estimate $f^\star$.
Throughout, we assume
\begin{itemize}
\item[(A1)] $X$ takes values in a bounded Lipschitz domain $\Omega\subseteq[-L,L]^d$, $d\geq2$, and has a density on $\Omega$ bounded by $\kappa>0$;
\item[(A2)] $\EE(\varepsilon\mid X)=0$ and $\EE(\varepsilon^2\mid X)\leq\sigma^2$, for some $\sigma>0$.
\end{itemize}

\paragraph{Physical model and smoothness.}
The target approximately satisfies a PDE $\D(f^\star)=g$ of integer order $r\geq1$: its residual $\|\D(f^\star)-g\|_{L^2(\Omega)}$ is at most $\delta$.
Except in \cref{thm:universal}, $\D(f)=\sum_{|\alpha|\leq r}a_\alpha\partial^\alpha f$ is linear, where $\alpha\in\mathbb N^d$, $|\alpha|=\alpha_1+\cdots+\alpha_d$ and $\partial^\alpha=\partial_1^{\alpha_1}\cdots\partial_d^{\alpha_d}$; the coefficients $a_\alpha$ are bounded, with at least one nonzero coefficient of order $r$.
Until \cref{sec:source}, $g=0$.
We assume $f^\star\in H^s(\Omega)$ for a real order $s\geq r$, where, for $s\geq0$,
\[
\|f\|^2_{H^s(\Omega)}=\sum_{|\alpha|\leq\lfloor s\rfloor}\|\partial^\alpha f\|^2_{L^2(\Omega)}
+\sum_{|\alpha|=\lfloor s\rfloor}\int_{\Omega^2}\frac{|\partial^\alpha f(x)-\partial^\alpha f(y)|^2}{|x-y|^{d+2(s-\lfloor s\rfloor)}}\,dx\,dy
\]
\citep[Section~2]{dinezza2012hitchhiker}, the second term being omitted when $s$ is an integer.
For $\rho>0$ and $\delta\geq0$, let
\[
\Hc(\rho)=\{f:\|f\|_{H^s(\Omega)}\leq\rho\},
\qquad
\F(\rho,\delta)=\{f\in\Hc(\rho):\|\D(f)\|_{L^2(\Omega)}\leq\delta\}.
\]

\paragraph{Minimax risk.}
For a class $\mathcal C$, $\M_n(\mathcal C)=\inf_{\hat f}\sup_{f^\star\in\mathcal C}\EE\|\hat f-f^\star\|^2_{L^2(\PP_X)}$, where the infimum is over all estimators and the design and noise laws are fixed.
For uniform design, independent Gaussian noise of variance $\sigma^2$, $s>d/2$ and fixed $\rho,\sigma$, the Sobolev minimax risk has order $\rho^2(\sigma^2/(n\rho^2))^{\gamma(s,d)}$ \citep{gyorfi2002distribution}, where $\gamma(s,d)=2s/(2s+d)$.

\paragraph{Physics-informed kernel estimator.}
For an integer $m\geq1$, let
\[
\phi(x)=\big(e^{-i\pi\langle k,x\rangle/(2L)}\big)_{k\in\Z^d,\,\|k\|_\infty\leq m}
\]
be the Fourier map on the torus $[-2L,2L]^d$, and $f_\theta(x)=\phi(x)^\ast\theta$ for $\theta\in\mathbb C^{(2m+1)^d}$.
Following \citet{doumeche2024physicsinformed,doumeche2026fast}, let $\hat f_n=f_{\hat\theta}$, where
\begin{equation}
\hat\theta=\mathop{\mathrm{argmin}}_\theta\ \frac1n\sum_{i=1}^n|f_\theta(X_i)-Y_i|^2
+\lambda\|S\theta\|_2^2+\mu\|\D(f_\theta)\|^2_{L^2(\Omega)},
\qquad
S=\mathrm{diag}\big((1+\|k\|_2^{2s})^{1/2}\big),
\label{eq:estimator}
\end{equation}
with $\lambda>0$ and $\mu\geq0$; $\|S\theta\|_2$ is equivalent to the Sobolev norm on the torus, and $\mu=0$ gives the Sobolev kernel estimator.
For a constant-coefficient operator, let $P(z)=\sum_\alpha a_\alpha z^\alpha$, $D=\mathrm{diag}(P(i\pi k/(2L)))$, $G_{k,l}=\int_\Omega e^{i\pi\langle l-k,x\rangle/(2L)}dx$ and $\Sigma=\EE[\phi(X)\phi(X)^\ast]$.
Then $\D(f_\theta)=f_{D\theta}$, $\|\D(f_\theta)\|^2_{L^2(\Omega)}=\theta^\ast D^\ast GD\theta$, and the effective dimension is
\[
\mathcal N(\lambda,\mu)=\mathrm{tr}\big((\Sigma+\lambda S^2+\mu D^\ast GD)^{-1}\Sigma\big).
\]

\section{A universal lower bound for linear and some nonlinear PDEs}
\label{sec:universal}
For every linear PDE, and many nonlinear ones, the physics can only help when $\delta$ is small compared with
\[
\delta_n=\rho\,\Big(\frac{\sigma^2}{n\rho^2}\Big)^{(s-r)/(2s+d)},
\]
the order of the residual of the functions that are hardest to estimate in $\Hc(\rho)$ from $n$ samples.

\begin{theorem}[Universal lower bound]
\label{thm:universal}
Let $r\geq1$ be an integer, $s\geq r$ a real number with $s>d/2$, and $\D$ an operator, linear or not, such that, for some $A>0$, $\|\D(f)\|_{L^2(\Omega)}\leq A\|f\|_{H^r(\Omega)}$ for every $f\in H^s(\Omega)$ with $|\partial^\alpha f|\leq1$ on $\Omega$ for all $|\alpha|\leq r$.
Assume that $\PP_X$ is the uniform distribution on $\Omega$ and $\varepsilon\sim\mathcal N(0,\sigma^2)$ is independent of $X$.
There are $c,C_0>0$, depending only on $d$, $s$, $r$, $A$ and $\Omega$, such that, for all $\delta>0$ with $n\geq C_0\sigma^2/\delta^2$ and $n\geq C_0\sigma^2/\rho^2$, and, if $\D$ is nonlinear, with $s>r$ and $n\geq C_0\sigma^2\rho^{(d+2r)/(s-r)}$,
\[
\M_n(\F(\rho,\delta))\geq c\rho^2\Big(\frac{\sigma^2}{n\rho^2}\Big)^{\gamma(s,d)}
\min\Big\{1,\Big(\frac{\delta}{\delta_n}\Big)^{2d/(2r+d)}\Big\}.
\]
\end{theorem}

Every linear differential operator of order $r$ with bounded coefficients satisfies the assumption, whatever its type and even with variable coefficients, and so do Burgers' operator $\partial_tf+f\partial_{x'}f-\nu\partial^2_{x'}f$ ($d=2$), the Allen--Cahn operator $\partial_tf-\Delta_{x'}f+f^3-f$ and the Kardar--Parisi--Zhang operator $\partial_tf-\Delta_{x'}f-|\nabla_{x'}f|^2$ ($r=2$).
When $\delta\geq\delta_n$, the bound is of the order of the Sobolev minimax risk: knowing the residual bound improves it by at most a constant factor.
When $\delta<\delta_n$, it equals $c\delta^2(\sigma^2/(n\delta^2))^{\gamma(r,d)}$; for linear $\D$ with $A=(\sum_\alpha\sup|a_\alpha|^2)^{1/2}$, $\F(\rho,\delta)$ contains the functions of $\Hc(\rho)$ with Sobolev norm of order $r$ at most $\delta/A$, and the physics never helps more than such a bound would.
Since $\delta_n\to0$ when $s>r$, a fixed $\delta>0$ eventually exceeds $\delta_n$, which rules out~\eqref{eq:naive}.

\begin{cor}[No bound vanishing with $\delta$]
\label[cor]{cor:naive}
Under the assumptions of \cref{thm:universal} on $\D$, $\PP_X$ and $\varepsilon$, let $s>\max(r,d/2)$, and fix $\rho$ and $\sigma$.
For every function $h$ with $h(\delta)\to0$ as $\delta\to0$ and every $C>0$, no estimator $\hat f_n$ satisfies, for all $n\geq1$ and $\delta\in(0,1]$,
\[
\sup_{f^\star\in\F(\rho,\delta)}\EE\|\hat f_n-f^\star\|^2_{L^2(\PP_X)}
\leq C\big(h(\delta)n^{-\gamma(s,d)}+n^{-\gamma(s,d-1)}\big).
\]
\end{cor}

\section{Minimax rates over common PDE classes}
\label{sec:rates}

\subsection{Constant-coefficient elliptic equations}
\label{sec:elliptic}
A linear operator $\D$ of order $r$ with constant real coefficients is \emph{elliptic} if $\sum_{|\alpha|=r}a_\alpha\zeta^\alpha\neq0$ for every $\zeta\in\R^d\setminus\{0\}$, where $\zeta^\alpha=\zeta_1^{\alpha_1}\cdots\zeta_d^{\alpha_d}$; then $r$ is even.
Examples are the Laplace operator $\Delta$, the Helmholtz operator $\Delta+\omega^2$, the advection--diffusion operator $f\mapsto\nabla\cdot(\mathbf K\nabla f)+\langle v,\nabla f\rangle+b_0f$ with $\mathbf K$ symmetric positive definite, and the bilaplacian $\Delta^2$.

\begin{theorem}[Elliptic equations]
\label{thm:elliptic}
Let $\D$ be elliptic of order $r$ with constant real coefficients, assume $d<2r$, and let $s>r$ be real.
Assume (A1)--(A2).
There are $C,C_0>0$, depending only on $d$, $s$, $L$, $\kappa$, $\Omega$ and $\D$, such that the following holds for all $n\geq C_0\sigma^2/\rho^2$ with $(\sigma^2/(n\rho^2))^{\gamma(s,d-1)}\geq1/n$.
\begin{itemize}
\item[(a)] If $\delta\leq\delta_n$, the estimator~\eqref{eq:estimator} with
\begin{gather*}
\lambda=\max\Big\{\Big(\frac{\sigma^2}{n\rho^2}\Big)^{\gamma(s,d-1)},\frac{\delta^{2d/(2r+d)}}{\rho^2}\Big(\frac{\sigma^2}n\Big)^{\gamma(r,d)}\Big\},
\qquad
\mu=\min\Big\{\lambda^{r(d-1)/(sd)},\frac{\lambda\rho^2}{\delta^2}\Big\},\\
m\geq\max\Big\{\lambda^{-1/(2s)},\Big(\frac\mu\lambda\Big)^{1/(2s-2r)}\Big\}
\end{gather*}
(with $\lambda\rho^2/\delta^2=+\infty$ if $\delta=0$) has a worst-case risk over $\F(\rho,\delta)$ at most
\[
C\rho^2\Big(\frac{\sigma^2}{n\rho^2}\Big)^{\gamma(s,d-1)}+C\delta^{2d/(2r+d)}\Big(\frac{\sigma^2}n\Big)^{\gamma(r,d)}.
\]
\item[(b)] If $\delta>\delta_n$, the Sobolev kernel estimator ($\mu=0$) with $\lambda=(\sigma^2/(n\rho^2))^{\gamma(s,d)}$ and $m\geq\lambda^{-1/(2s)}$ has a worst-case risk over $\F(\rho,\delta)$ at most $C\rho^2(\sigma^2/(n\rho^2))^{\gamma(s,d)}$.
\end{itemize}
\end{theorem}

The condition $(\sigma^2/(n\rho^2))^{\gamma(s,d-1)}\geq1/n$ is equivalent to $n\geq(\rho/\sigma)^{4s/(d-1)}$, and $d<2r$ means $d\in\{2,3\}$ for second-order operators.
With \cref{thm:universal}, elliptic equations attain the universal limit, up to the Sobolev rate in dimension $d-1$.
This term is necessary on domains with a \emph{flat face}, that is, when there are a unit vector $\mathbf e$, $t_0\in\R$, $a_0>0$ and a nonempty open set $V$ of the hyperplane $\{\langle\mathbf e,x\rangle=t_0\}$ such that $\Omega\subseteq\{\langle\mathbf e,x\rangle<t_0\}$ and $V-t\mathbf e\subseteq\Omega$ for $t\in(0,a_0)$, as for convex polytopes.

\begin{prop}[Exact solutions of elliptic equations]
\label[prop]{prop:ellexact}
Let $\D$ be elliptic of order $r$ with constant real coefficients, $s\geq r$ real, and assume that $\Omega$ has a flat face, that $\PP_X$ is uniform on $\Omega$ and that $\varepsilon\sim\mathcal N(0,\sigma^2)$ is independent of $X$.
There are $c,C_0>0$, depending only on $d$, $s$, $\Omega$ and $\D$, such that, for $n\geq C_0\sigma^2/\rho^2$,
\[
\M_n(\F(\rho,0))\geq c\rho^2\Big(\frac{\sigma^2}{n\rho^2}\Big)^{\gamma(s,d-1)}.
\]
\end{prop}

Hence, on domains with a flat face, for uniform design and Gaussian noise and under the conditions of \cref{thm:elliptic} (possibly after enlarging $C_0$), $\M_n(\F(\rho,\delta))$ is of the order of
\[
\rho^2\Big(\frac{\sigma^2}{n\rho^2}\Big)^{\gamma(s,d-1)}
+\min\Big\{\rho^2\Big(\frac{\sigma^2}{n\rho^2}\Big)^{\gamma(s,d)},\delta^{2d/(2r+d)}\Big(\frac{\sigma^2}n\Big)^{\gamma(r,d)}\Big\},
\]
without logarithmic factors.

\subsection{Constant-coefficient hyperbolic equations}
\label{sec:hyperbolic}
We consider linear operators $\D(f)=\sum_{i,j=1}^da_{ij}\partial_i\partial_jf+\langle b,\nabla f\rangle+b_0f$ with constant real coefficients, $(a_{ij})$ symmetric, $b\in\R^d$ and $b_0\in\R$.
Such an operator is \emph{hyperbolic} of order $r=1$ if $(a_{ij})=0$ and $b\neq0$, and of order $r=2$ if $(a_{ij})$ or $-(a_{ij})$ has one positive and $d-1$ negative eigenvalues.
Examples include the transport operator $\partial_t+\langle v,\nabla_{x'}\rangle$, the wave operator $\partial_t^2-\Delta_{x'}$, the Klein--Gordon operator $\partial_t^2-\Delta_{x'}+b_0$ and the telegraph operator $\partial_t^2+b_1\partial_t-\Delta_{x'}$.

\begin{theorem}[Hyperbolic equations]
\label{thm:hyperbolic}
Let $\D$ be hyperbolic of order $r\in\{1,2\}$ with constant real coefficients, and let $s\geq r$ be real.
There are $c,C,C_0>0$, depending only on $d$, $s$, $L$, $\kappa$, $\Omega$ and $\D$, such that the following bounds hold for all $n\geq C_0\sigma^2/\rho^2$.

\textbf{Lower bound:} If $\PP_X$ is uniform on $\Omega$ and $\varepsilon\sim\mathcal N(0,\sigma^2)$ is independent of $X$, then, for all $\delta\geq0$, $\M_n(\F(\rho,\delta))$ is at least $c\rho^2$ times
\[
\Big(\frac{\sigma^2}{n\rho^2}\Big)^{\gamma(s,d-1)}
+\min\Big\{\Big(\frac{\sigma^2}{n\rho^2}\Big)^{\gamma(s,d)},\Big(\frac{\sigma^2\delta}{n\rho^3}\Big)^{2s/(3s+d-r)}\Big\}.
\]

\textbf{Upper bound:} Assume (A1)--(A2), $s>\max(r,d-r)$ and $(\sigma^2/(n\rho^2))^{\gamma(s,d-1)}\geq1/n$, and let $\ell=\log(2+n\rho^2/\sigma^2)$ if $r=d=2$ and $\ell=1$ otherwise.
\begin{itemize}
\item[(a)] If $\delta\ell\leq\delta_n$, the estimator~\eqref{eq:estimator} with
\begin{gather*}
\lambda=\max\Big\{\Big(\frac{\sigma^2}{n\rho^2}\Big)^{\gamma(s,d-1)},\Big(\frac{\sigma^2\delta\ell}{n\rho^3}\Big)^{2s/(3s+d-r)}\Big\},
\qquad
\mu=\min\Big\{\lambda^{(r-1)/s},\frac{\lambda\rho^2}{\delta^2}\Big\},\\
m\geq\max\Big\{\lambda^{-1/(2s)},\Big(\frac\mu\lambda\Big)^{1/(2s-2r)}\Big\}
\end{gather*}
(with $\lambda\rho^2/\delta^2=+\infty$ if $\delta=0$) has a worst-case risk over $\F(\rho,\delta)$ at most $C\rho^2$ times
\[
\log\Big(2+\frac{n\rho^2}{\sigma^2}\Big)\Big(\frac{\sigma^2}{n\rho^2}\Big)^{\gamma(s,d-1)}+\Big(\frac{\sigma^2\delta\ell}{n\rho^3}\Big)^{2s/(3s+d-r)}.
\]
\item[(b)] If $\delta\ell>\delta_n$, the Sobolev kernel estimator ($\mu=0$) with $\lambda=(\sigma^2/(n\rho^2))^{\gamma(s,d)}$ and $m\geq\lambda^{-1/(2s)}$ has a worst-case risk over $\F(\rho,\delta)$ at most $C\rho^2(\sigma^2/(n\rho^2))^{\gamma(s,d)}$.
\end{itemize}
\end{theorem}

The first term is the Sobolev rate in dimension $d-1$: without physical error, the problem is $(d-1)$-dimensional.
The minimum equals $\rho^2(\sigma^2/(n\rho^2))^{\gamma(s,d)}\min\{1,(\delta/\delta_n)^{2s/(3s+d-r)}\}$.
Since $2s/(3s+d-r)<2d/(2r+d)$ when $r<d$, hyperbolic equations do not attain the universal limit, whereas for $r=d=2$ the residual term attains it up to the factor $\ell$.
For fixed $\rho,\sigma$, $s>\max(r,d-r)$ and $\delta=n^{-\tau}$, the rate is $n^{-e}$ up to logarithmic factors, where $e$ is $\gamma(s,d)$ for $\tau\leq(s-r)/(2s+d)$, $\gamma(s,d-1)$ for $\tau\geq(s-r+1)/(2s+d-1)$, and linear in $\tau$ in between.
\Cref{tab:rates} summarizes the rates.

\begin{table}[t]
\centering
\caption{Rates for fixed $\rho,\sigma>0$, uniform design and Gaussian noise, under the stated conditions and for large $n$; the elliptic row holds on domains with a flat face, and the hyperbolic row up to logarithmic factors.}
\label{tab:rates}
\small
\setlength{\tabcolsep}{3.5pt}
\begin{tabular}{@{}lll@{}}
\toprule
Class & Conditions & Rate\\
\midrule
Sobolev ball & $s>d/2$ & $n^{-2s/(2s+d)}$\\[3pt]
Elliptic & $d<2r$, $s>r$ & $n^{-2s/(2s+d-1)}+\min\{n^{-2s/(2s+d)},\delta^{2d/(2r+d)}n^{-2r/(2r+d)}\}$\\[3pt]
Hyperbolic & $r\in\{1,2\}$, $s>\max(r,d-r)$ & $n^{-2s/(2s+d-1)}+\min\{n^{-2s/(2s+d)},(\delta/n)^{2s/(3s+d-r)}\}$\\
\bottomrule
\end{tabular}
\end{table}

\subsection{Linear PDEs with known source terms}
\label{sec:source}
Let $\D$ be linear with real coefficients, let $g\in L^2(\Omega)$ be real-valued, let $\F_g(\rho,\delta)=\{f\in\Hc(\rho):\|\D(f)-g\|_{L^2(\Omega)}\leq\delta\}$, and let $\hat f_n$ be~\eqref{eq:estimator} with $\|\D(f_\theta)-g\|^2_{L^2(\Omega)}$ in place of $\|\D(f_\theta)\|^2_{L^2(\Omega)}$.

\begin{prop}[Source terms]
\label[prop]{prop:source}
\begin{itemize}
\item[(i)] If $\D(f_0)=g$ for some $f_0\in\Hc(\rho/2)$, then
\[
\M_n(\F_g(\rho,\delta))\geq\M_n(\F(\rho/2,\delta)).
\]
\item[(ii)] There is $C>0$, depending only on $d,s,L,\kappa,\Omega$ and $\D$, such that under (A1)--(A2), for constant-coefficient $\D$, $s>\max(r,d/2)$, $\lambda\in[1/n,1]$, $\mu\geq0$ and
\[
m\geq\max\{\lambda^{-1/(2s)},(\mu/\lambda)^{1/(2s-2r)}\},
\]
\[
\sup_{f^\star\in\F_g(\rho,\delta)}\EE\|\hat f_n-f^\star\|^2_{L^2(\PP_X)}
\leq C\Big(\lambda\rho^2+\mu\delta^2+\frac{\sigma^2}n\mathcal N(\lambda,\mu)\Big).
\]
\end{itemize}
Consequently, if $f_0$ exists, the lower bounds of \cref{thm:universal,thm:hyperbolic,prop:ellexact} hold over $\F_g(\rho,\delta)$ with $\rho/2$ in place of $\rho$, and the upper bounds of \cref{thm:elliptic,thm:hyperbolic} hold over $\F_g(\rho,\delta)$ with the same parameters.
\end{prop}

\subsection{The effective dimension of the physics-informed kernel}
\label{sec:count}
The upper bounds follow from an oracle inequality for~\eqref{eq:estimator} and from the following bounds.

\begin{theorem}[Effective dimension]
\label{thm:count}
Assume (A1), and let $\D$ be linear with constant real coefficients.
For $\lambda\in(0,1]$, let $\ell_{r,d}(\lambda)=\log(2+1/\lambda)$ if $r=d=2$ and $\ell_{r,d}(\lambda)=1$ otherwise.
There is $C>0$, depending only on $d,s,L,\kappa,\Omega$ and $\D$, and independent of $\lambda\in(0,1]$, $\mu\geq0$ and $m$, such that:
\begin{itemize}
\item[(i)] if $s>d/2$, then $\mathcal N(\lambda,0)\leq C\lambda^{-d/(2s)}$;
\item[(ii)] if $\D$ is hyperbolic of order $r\in\{1,2\}$, $s>\max(r,d-r)$ and $\mu\in(0,1]$, then $\mathcal N(\lambda,\mu)\leq C\big(\lambda^{-(d-1)/(2s)}\log(2+1/\lambda) +\lambda^{-(d-r)/(2s)}\mu^{-1/2}\ell_{r,d}(\lambda)\big)$;
\item[(iii)] if $\D$ is elliptic of order $r$, $d<2r$, $s\geq r$ and $\mu>0$, then $\mathcal N(\lambda,\mu)\leq C\big(\lambda^{-(d-1)/(2s)}+\mu^{-d/(2r)}\big)$.
\end{itemize}
\end{theorem}
Since $\mathcal N(\lambda,\mu)\leq\mathcal N(\lambda,0)$, bound (i) can be used whenever it is smaller.

\section{Implementation}
\label{sec:implementation}

\subsection{Hyperparameter tuning with an unknown modeling error}
\label{sec:adaptation}
The residual bound $\delta$ can be removed from the tuning by validation, under a stronger noise assumption; the construction still uses $s$ and a known radius $\rho$.
Split the $n\geq4$ observations into a training sample of size $n_1=\lfloor n/2\rfloor$ and a validation sample of size $n_2=n-n_1$, and let $\mathcal G=\{2^j/n_1:j=0,\ldots,\lfloor\log_2n_1\rfloor\}\cup\{1\}$.
For $\lambda\in\mathcal G$ and $\mu\in\mathcal G\cup\{0\}$, fit~\eqref{eq:estimator} on the training sample with $m_{\lambda,\mu}=\lceil\max\{\lambda^{-1/(2s)},(\mu/\lambda)^{1/(2s-2r)}\}\rceil$, and clip it to $[-C_{\mathrm{emb}}\rho,C_{\mathrm{emb}}\rho]$, where $\|f\|_{L^\infty(\Omega)}\leq C_{\mathrm{emb}}\|f\|_{H^s(\Omega)}$; call the result $\check f_{\lambda,\mu}$.
Let $\tilde f_n=\check f_{\hat\lambda,\hat\mu}$, where $(\hat\lambda,\hat\mu)$ minimizes the validation error $n_2^{-1}\sum_{i>n_1}(Y_i-\check f_{\lambda,\mu}(X_i))^2$, with ties broken by a fixed order.

\begin{theorem}[Adaptation to the residual bound]
\label{thm:adaptation}
Assume (A1) and, for some $\sigma>0$, $\EE[\exp(t\varepsilon)\mid X]\leq\exp(\sigma^2t^2/2)$ for every $t\in\R$.
Assume the operator and smoothness conditions of \cref{thm:elliptic} or of the upper bound of \cref{thm:hyperbolic}.
There are $C,C_0>0$, depending only on the quantities specified in those theorems, such that, whenever $n_1\geq C_0\sigma^2/\rho^2$ and $(\sigma^2/(n_1\rho^2))^{\gamma(s,d-1)}\geq1/n_1$, for every $\delta\geq0$, the estimator $\tilde f_n$, whose construction does not use $\delta$, has a worst-case risk over $\F(\rho,\delta)$ bounded by the upper bounds of \cref{thm:elliptic} or \cref{thm:hyperbolic}, respectively, with a larger constant $C$ and the additional term $C(\rho^2+\sigma^2)\log\log(3+n)/n$.
\end{theorem}

The sub-Gaussian bound implies (A2); the construction does not use $\sigma$, and for fixed $\rho,\sigma$ its validation term is of smaller order than the exact-solution rate.

\subsection{A fast algorithm attaining the optimal rate}
\label{sec:algo}
For a constant-coefficient operator, the minimizer of~\eqref{eq:estimator} solves $A_n\theta=v_n$, where
\[
A_n=\hat\Sigma+\lambda S^2+\mu D^\ast GD,\qquad
\hat\Sigma=\frac1n\sum_{i=1}^n\phi(X_i)\phi(X_i)^\ast,
\qquad v_n=\frac1n\sum_{i=1}^nY_i\phi(X_i).
\]
The matrices $\hat\Sigma$ and $G$ are $d$-level Toeplitz: their entries depend only on $l-k$, with $\|l-k\|_\infty\leq2m$.

\begin{algorithm}[t]
\caption{Fast physics-informed kernel regression}
\label{alg:fast}
\begin{algorithmic}
\STATE \textbf{Input:} data $(X_i,Y_i)_{i\leq n}$, $m,\lambda,\mu$, number of iterations $K$.
\STATE Compute the entries $n^{-1}\sum_ie^{i\pi\langle j,X_i\rangle/(2L)}$ ($\|j\|_\infty\leq2m$) of $\hat\Sigma$, and $v_n$, by type-1 non-uniform FFT (NUFFT).
\STATE Compute $G_{0,j}$ ($\|j\|_\infty\leq2m$): closed form for boxes and polytopes, or FFT of $\mathbf 1_\Omega$ on a grid.
\STATE Starting from $\theta^{(0)}=0$, run at most $K$ conjugate-gradient iterations on $A_n\theta=v_n$, preconditioned by $\lambda S^2$, with FFT-based Toeplitz products; terminate if the residual is zero.
\STATE \textbf{Output:} $f_\theta$, evaluated at new points by type-2 NUFFT.
\end{algorithmic}
\end{algorithm}

\begin{theorem}[Complexity]
\label{thm:algo}
Let $\D$ be linear of order $r\leq s$ with constant real coefficients, and assume that the Toeplitz generating sequences $\hat\Sigma_{0,j}$ and $G_{0,j}$, $\|j\|_\infty\leq2m$, and $v_n$ are supplied exactly.
Each iteration of \cref{alg:fast} costs $O(m^d\log(2+m))$ arithmetic operations.
Under the conditions of \cref{thm:elliptic} or of the upper bound of \cref{thm:hyperbolic}, with the corresponding $(\lambda,\mu,m)$ and fixed $\rho,\sigma$, there is $C_K>0$, depending only on $d$, $s$, $L$, $\kappa$, $\Omega$, $\D$, $\rho$ and $\sigma$, such that $K\geq C_K\lambda^{-1/2}\log(2+n)$ iterations give the upper bounds of those theorems with $C$ replaced by $2C+2$.
If $K=\lceil C_K\lambda^{-1/2}\log(2+n)\rceil$, $m$ is the smallest allowed truncation, and $(s-1)(s-r)\geq d$ for hyperbolic operators or $s\geq r+1$ for elliptic operators, the iterations cost $O(n\log^2(2+n))$ operations in total.
\end{theorem}

At fixed relative accuracy, NUFFT \citep{barnett2019parallel} computes $\hat\Sigma$ and $v_n$ in $O(n+m^d\log(2+m))$ operations, and $G$ has a closed form for boxes and polytopes; numerical errors and the computation of $G$ on a general domain are outside the theorem.

\subsection{Numerical illustration}
We take the wave operator $\partial_t^2-\partial_x^2$ on $\Omega=(-1/2,1/2)^2$, uniform design, Gaussian noise with $\sigma=0.5$, $s=3$, and exact solutions given by finite Fourier sums along the two characteristics; for $d=2$ and $s=3$, the smoothness conditions of \cref{thm:algo} hold for the wave, transport and Laplace operators.
The parameters follow the powers of $n$ in \cref{thm:hyperbolic,thm:elliptic}, with constants set to one (details in \cref{app:exp}).
With exact physics and ten resamples for $n\in[10^2,10^7]$, the least-squares slopes of the mean test error are $0.77$ for the Sobolev estimator and $0.89$ for the physics-informed one, and their error ratio decreases to $0.19$ at $n=10^7$ (\cref{fig:exp}, left); for fixed smooth targets, slopes may exceed the minimax exponents $0.75$ and $6/7$.
Transport, Laplace and three-dimensional wave experiments behave similarly (\cref{app:exp}).
For imperfect physics, we add to this target a rough random perturbation whose residual gives $\delta/\delta_n\in\{0.1,1\}$ at $n=10^4$, and select each estimator's parameters by the smallest test error over a grid, an oracle choice (\cref{fig:exp}, middle).
The error ratio decreases from $0.83$ to below $0.5$ with exact physics, stays around $0.8$ for $\delta/\delta_n=0.1$, and is close to one for $\delta/\delta_n=1$, as \cref{thm:universal} suggests, although single targets do not determine the worst-case threshold; for $n\geq3\cdot10^5$, the grid limits the oracle (\cref{app:exp}).
Starting from $\theta^{(0)}=0$, a fit with $n=10^7$ points takes $5.6$ seconds on one core of an Apple M2 Pro, and the number of iterations increases from $15$ to $431$ over $n\in[10^2,10^7]$ (fitted exponent $0.29$, below the order $\lambda^{-1/2}=n^{3/7}$ of \cref{thm:algo}).

\begin{figure}[t]
\centering
\includegraphics[width=0.32\textwidth]{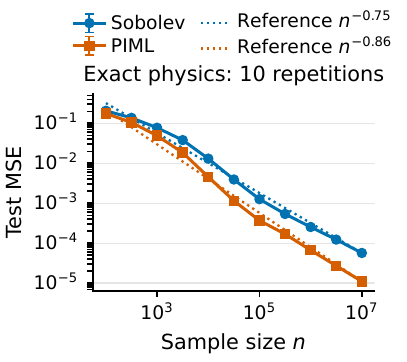}
\includegraphics[width=0.32\textwidth]{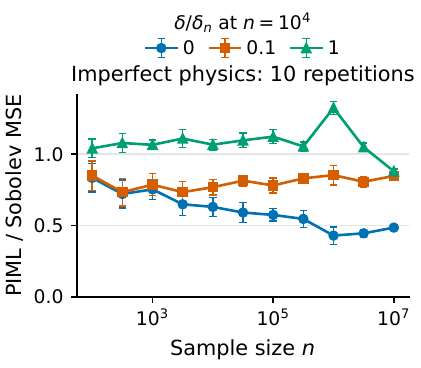}
\includegraphics[width=0.32\textwidth]{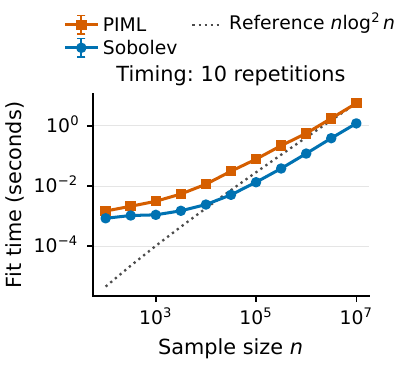}
\caption{Wave operator, $d=2$, $s=3$; means over ten resamples with $\pm2$ estimated standard errors. Left: test errors, exact physics. Middle: ratio of mean errors (physics-informed over Sobolev) after oracle selection, labeled by $\delta/\delta_n$ at $n=10^4$. Right: fit time of \cref{alg:fast} on one core.}
\label{fig:exp}
\end{figure}

\section{Conclusion}
\label{sec:discussion}
The relative residual $\delta/\delta_n$ controls the worst-case benefit of an approximate PDE prior: above this scale the Sobolev rate remains necessary, and below it elliptic and hyperbolic equations have different residual-dependent rates.
The conditions $d<2r$ and $s>\max(r,d-r)$ are sufficient for our analysis; their necessity is not established, and exact-solution lower bounds for elliptic operators on domains without a flat face are not treated here.
All lower bounds extend, with changed constants, to design densities bounded above and below on $\Omega$, and to noise independent of $X$ from location families $Q_u$ with $\mathrm{KL}(Q_u,Q_v)\leq C(u-v)^2/\sigma^2$, since only the constants $b$, $B_0$ and $a^2$ of \cref{lem:testing} change.
Parabolic, variable-coefficient and nonlinear equations are covered only by the universal lower bound; other residual norms and adaptation to smoothness are left for future work.
\label{endmain}

\appendix
\label{startapp}
\startcontents[appendix]
\printcontents[appendix]{}{1}[2]{\section*{Contents of the appendix}}

\section{Lower bounds}
\label{app:lower}
Throughout this appendix, $\varepsilon\sim\mathcal N(0,\sigma^2)$ is independent of $X$ and, except in \cref{lem:testing}, $\PP_X$ is the uniform distribution on $\Omega$, so that $\|g\|^2_{L^2(\PP_X)}=\|g\|^2_{L^2(\Omega)}/|\Omega|$.

\subsection{Notation and Fourier setting}
\label{app:tools}
Throughout the appendices, $C,c>0$ may change from line to line and depend only on $d$, $s$, $r$, $L$, $\kappa$, $\Omega$ and $\D$; any additional dependence is stated locally.

\paragraph{Fourier model.}
Write $e_k(x)=e^{i\pi\langle k,x\rangle/(2L)}$, so that $\phi=(\overline{e_k})_{\|k\|_\infty\leq m}$ and $f_\theta=\sum_k\theta_ke_k$.
For constant-coefficient $\D$, $\D(f_\theta)=f_{D\theta}$ and $\|\D(f_\theta)\|^2_{L^2(\Omega)}=(D\theta)^\ast G(D\theta)$, and $P_r(z)=\sum_{|\alpha|=r}a_\alpha z^\alpha$ is the principal symbol.
The estimator has coefficients $\hat\theta=A_n^{-1}v_n$, with $A_n$ and $v_n$ as in \cref{sec:algo}.
Its objective is strictly convex and invariant under $\theta_k\mapsto\overline{\theta_{-k}}$, because the data and the coefficients of $\D$ are real, $e_{-k}=\overline{e_k}$ and $S$ is even in $k$; uniqueness gives $\hat\theta_{-k}=\overline{\hat\theta_k}$, so $f_{\hat\theta}$ is real-valued.

\paragraph{Periodic functions.}
For $q\in\mathbb N^\ast$ and $L>0$, let $\T^q_L=\R^q/(4L\Z)^q$, identified with $[-2L,2L]^q$ with periodic boundary conditions, and let $\T=\T^d_L$ with the $L$ of (A1).
For $g\in L^2(\T^q_L)$ and $k\in\Z^q$, set $\hat g_k=(4L)^{-q}\int_{\T^q_L}g\,\overline{e_k}$ and $\|g\|^2_{H^s(\T^q_L)}=(4L)^q\sum_{k\in\Z^q}(1+\|\pi k/(2L)\|_2^2)^s|\hat g_k|^2$.
Since $c(1+\|k\|_2^{2s})\leq(1+\|\pi k/(2L)\|_2^2)^s\leq C(1+\|k\|_2^{2s})$, the norm on $\T$ is equivalent to $\|S\hat g\|_2$, with $S$ extended to all frequencies.
If $g_m$ keeps the frequencies $\|k\|_\infty\leq m$ of $g\in H^s(\T^q_L)$, then $\|g-g_m\|_{H^t(\T^q_L)}\leq Cm^{t-s}\|g\|_{H^s(\T^q_L)}$ for $0\leq t\leq s$.
For periodic $v$, every constant-coefficient operator of order $r$ satisfies $\|\D(v)\|_{L^2(\Omega)}\leq C\|v\|_{H^r(\T)}$.
When $s>d/2$, the constant $a_s=\sum_{k\in\Z^d}(1+\|k\|_2^{2s})^{-1}$ is finite and $\sup_x\phi(x)^\ast(\lambda S^2)^{-1}\phi(x)\leq a_s/\lambda$.

\paragraph{Other domains.}
Sobolev norms on an open set $U\subseteq\R^q$ are defined by the formula of \cref{sec:setting} over $U$, with $|x-y|^{q+2(s-\lfloor s\rfloor)}$ in the double integrals; $|v|_{s,U}$ denotes the seminorm given by $|v|^2_{s,U}=\sum_{|\alpha|=s}\|\partial^\alpha v\|^2_{L^2(U)}$ if $s$ is an integer, and by the double integrals of the norm otherwise; it vanishes on $\Pol_{\lceil s\rceil}$.
For $F\in L^1(\R^q)$, $\hat F(\zeta)=\int F(x)e^{-i\langle\zeta,x\rangle}dx$, extended to $L^2$ and to distributions as usual, and $\|F\|^2_{(t)}=\int(1+|\zeta|^2)^t|\hat F(\zeta)|^2d\zeta$ for $t\in\R$.
For $s\geq0$, $\|F\|_{(s)}$ is equivalent to $\|F\|_{H^s(\R^q)}$, by Plancherel's theorem applied to the derivatives of order $\lfloor s\rfloor$ and \citet[Proposition~3.4]{dinezza2012hitchhiker}.
If $U\subseteq\R^q$ is a bounded Lipschitz domain and $t\geq0$ is real, there is a linear operator $E$ with $Ef=f$ on $U$ and $\|Ef\|_{H^t(\R^q)}\leq C_{U,t}\|f\|_{H^t(U)}$, where $H^t(U)$ carries the norm of \cref{sec:setting} \citep[Definition~1.3.2.1 and Theorem~1.4.3.1]{grisvard1985elliptic}.
Finally, $\Pol_j:=\mathbb C_{j-1}[x_1,\ldots,x_d]$ denotes the polynomials of degree less than $j$, and a set $\Lambda\subset\R^q$ is \emph{$\tau$-separated} if $|\nu-\nu'|\geq\tau$ for all distinct $\nu,\nu'\in\Lambda$.

\subsection{Assouad's lemma for PDE-constrained classes}
All lower bounds follow from the following form of Assouad's lemma.

\begin{lem}[Assouad's lemma for testing families]
\label[lem]{lem:testing}
Let $\PP_X$ be a distribution on $\Omega$.
Let $\psi_1,\ldots,\psi_M\in H^s(\Omega)$ be real-valued, and let $b,B_0>0$ satisfy $\|\sum_jc_j\psi_j\|^2_{L^2(\PP_X)}\geq b\|c\|_2^2$ for $c\in\{-1,0,1\}^M$ and $\|\psi_j\|^2_{L^2(\PP_X)}\leq B_0$.
Let $a^2=16\sigma^2/(9nB_0)$ and $f_\omega=a\sum_j\omega_j\psi_j$ for $\omega\in\{0,1\}^M$.
\begin{itemize}
\item[(i)] If $f_\omega\in\F(\rho,\delta)$ for every $\omega\in\{0,1\}^M$, then $\M_n(\F(\rho,\delta))\geq\frac2{27}\frac b{B_0}\frac{\sigma^2M}n$.
\item[(ii)] The condition of (i) is satisfied if $\D$ is linear and there are $A_s,A_\D\geq0$ such that $\|\sum_jv_j\psi_j\|^2_{H^s(\Omega)}\leq A_s\|v\|_2^2$ and $\|\D(\sum_jv_j\psi_j)\|^2_{L^2(\Omega)}\leq A_\D\|v\|_2^2$ for $v\in\{0,1\}^M$, and if $16\sigma^2MA_s\leq9nB_0\rho^2$ and $16\sigma^2MA_\D\leq9nB_0\delta^2$.
\end{itemize}
\end{lem}
\begin{proof}
(i) Let $\hat f$ be an estimator, let $\EE_\omega$ be the expectation when $f^\star=f_\omega$, and let $\hat\omega$ minimize $\|\hat f-f_\omega\|_{L^2(\PP_X)}$ over $\omega\in\{0,1\}^M$.
Since $f_{\hat\omega}-f_\omega=a\sum_j(\hat\omega_j-\omega_j)\psi_j$ with $\hat\omega-\omega\in\{-1,0,1\}^M$, the triangle inequality gives $a^2b\,\mathrm{Ham}(\hat\omega,\omega)\leq\|f_{\hat\omega}-f_\omega\|^2_{L^2(\PP_X)}\leq4\|\hat f-f_\omega\|^2_{L^2(\PP_X)}$.
For neighbours $\omega,\omega'$, the Kullback--Leibler divergence between the sample laws is $\frac n{2\sigma^2}\|f_\omega-f_{\omega'}\|^2_{L^2(\PP_X)}\leq\frac{na^2B_0}{2\sigma^2}=\frac89$, so \citet[Theorem~2.12(iv)]{tsybakov2009introduction} gives $\max_\omega\EE_\omega\mathrm{Ham}(\hat\omega,\omega)\geq\frac M2(1-\sqrt{4/9})=\frac M6$.
Hence $\max_\omega\EE_\omega\|\hat f-f_\omega\|^2_{L^2(\PP_X)}\geq a^2bM/24=\frac2{27}\frac b{B_0}\frac{\sigma^2M}n$.
Since the maximum over $\omega$ is at most the supremum over $\F(\rho,\delta)$, and $\hat f$ is arbitrary, this proves (i).

(ii) By linearity, $\D(f_\omega)=a\D(\sum_j\omega_j\psi_j)$, and $\|\omega\|_2^2\leq M$ gives $\|f_\omega\|^2_{H^s(\Omega)}\leq a^2A_sM\leq\rho^2$ and $\|\D(f_\omega)\|^2_{L^2(\Omega)}\leq a^2A_\D M\leq\delta^2$.
\end{proof}

\subsection{Testing family for general operators}
\begin{lem}[Bumps]
\label[lem]{lem:bumps}
Let $Q\subseteq\Omega$ be a cube of side $\ell_Q\leq1$, let $\psi\in C_c^\infty((-1/2,1/2)^d)$ be real-valued and nonzero, let $M\geq1$ be an integer, $h=\ell_Q/M$, and $\psi_j(x)=\psi((x-y_j)/h)$, where $y_1,\ldots,y_{M^d}$ are the centers of the $M^d$ subcubes of $Q$ of side $h$.
There is $C_\psi>0$, depending only on $\psi$, $d$, $s$ and $r$, such that, for every $v\in\R^{M^d}$ and $t\in\{r,s\}$,
\begin{gather*}
\Big\|\sum_jv_j\psi_j\Big\|^2_{L^2(\PP_X)}=\frac{h^d\|\psi\|^2_{L^2(\R^d)}}{|\Omega|}\|v\|_2^2,
\qquad
\Big\|\sum_jv_j\psi_j\Big\|^2_{H^t(\Omega)}\leq C_\psi h^{d-2t}\|v\|_2^2,\\
\max_{|\alpha|\leq r}\Big\|\partial^\alpha\sum_jv_j\psi_j\Big\|_{L^\infty}\leq C_\psi h^{-r}\|v\|_\infty.
\end{gather*}
\end{lem}
\begin{proof}
The supports of the $\psi_j$ are disjoint, so at each point at most one term of the sum is nonzero; this gives the identity and, with $h\leq1$, the bound on derivatives.
For integer $t$, the squared $H^t(\R^d)$ norms of the terms therefore add.
For noninteger $t$ and $|\beta|=\lfloor t\rfloor$, at most two of the differences $\partial^\beta\psi_j(x)-\partial^\beta\psi_j(y)$ are nonzero for each pair $(x,y)$, so $|\sum_jv_j(\partial^\beta\psi_j(x)-\partial^\beta\psi_j(y))|^2\leq2\sum_jv_j^2|\partial^\beta\psi_j(x)-\partial^\beta\psi_j(y)|^2$, and integration gives $\|\sum_jv_j\psi_j\|^2_{H^t(\R^d)}\leq2\sum_jv_j^2\|\psi_j\|^2_{H^t(\R^d)}$.
The change of variables $x\mapsto y_j+hx$ multiplies $\|\partial^\alpha\psi\|^2_{L^2(\R^d)}$ by $h^{d-2|\alpha|}$ and the double integral by $h^{d-2t}$, so $\|\psi_j\|^2_{H^t(\R^d)}\leq h^{d-2t}\|\psi\|^2_{H^t(\R^d)}$ since $h\leq1$.
Restriction from $\R^d$ to $\Omega$ does not increase the norm, which gives the $H^t(\Omega)$ bound.
\end{proof}

\subsection{Proof of \texorpdfstring{\cref{thm:universal}}{the universal lower bound}}
\begin{proof}
Fix $Q$ and $\psi$ as in \cref{lem:bumps}, let $M\geq1$ be an integer, and let $f_\omega=a\sum_j\omega_j\psi_j$, with the functions $\psi_j$ of \cref{lem:bumps}, $b=B_0=h^d\|\psi\|^2_{L^2(\R^d)}/|\Omega|$ and $a^2=16\sigma^2/(9nB_0)$.
Since $a^2\leq C\sigma^2h^{-d}/n$, $\|\omega\|_2^2\leq M^d$, $\|\omega\|_\infty\leq1$ and $h^{-1}=M/\ell_Q$, \cref{lem:bumps} gives a constant $C_1\geq1$ such that
\begin{gather*}
\|f_\omega\|^2_{H^s(\Omega)}\leq C_1\frac{\sigma^2}nM^{2s+d},\qquad
\|f_\omega\|^2_{H^r(\Omega)}\leq C_1\frac{\sigma^2}nM^{2r+d},\\
\max_{|\alpha|\leq r}\|\partial^\alpha f_\omega\|_{L^\infty}\leq C_1\Big(\frac{\sigma^2}n\Big)^{1/2}M^{(2r+d)/2}.
\end{gather*}
Let $M_s=(n\rho^2/(C_1\sigma^2))^{1/(2s+d)}$, $M_r=(n\delta^2/(C_1A^2\sigma^2))^{1/(2r+d)}$ and $M=\lfloor\min(M_s,M_r)\rfloor$.
With $C_0\geq C_1\max(1,A^2)$, the sample-size assumptions give $\min(M_s,M_r)\geq1$, hence $M\geq\min(M_s,M_r)/2$, $\|f_\omega\|_{H^s(\Omega)}\leq\rho$ and $A\|f_\omega\|_{H^r(\Omega)}\leq\delta$.
Moreover, $M\leq M_s$ gives $\max_{|\alpha|\leq r}\|\partial^\alpha f_\omega\|_{L^\infty}\leq C\rho(\sigma^2/(n\rho^2))^{(s-r)/(2s+d)}$.
If $\D$ is nonlinear, this is at most $1$ when $n\geq C_0\sigma^2\rho^{(d+2r)/(s-r)}$ with $C_0$ large enough, and the assumption on $\D$ gives $\|\D(f_\omega)\|_{L^2(\Omega)}\leq A\|f_\omega\|_{H^r(\Omega)}$.
If $\D$ is linear, the same inequality follows by applying the assumption to $f_\omega/\max(1,\max_{|\alpha|\leq r}\|\partial^\alpha f_\omega\|_{L^\infty})$.
Hence every $f_\omega$ lies in $\F(\rho,\delta)$, and \cref{lem:testing}(i) gives
\begin{align*}
\M_n(\F(\rho,\delta))\geq\frac{2\sigma^2M^d}{27n}
&\geq c\frac{\sigma^2}n\min\Big\{\Big(\frac{n\rho^2}{\sigma^2}\Big)^{d/(2s+d)},\Big(\frac{n\delta^2}{\sigma^2}\Big)^{d/(2r+d)}\Big\}\\
&=c\min\Big\{\rho^2\Big(\frac{\sigma^2}{n\rho^2}\Big)^{\gamma(s,d)},\delta^{2d/(2r+d)}\Big(\frac{\sigma^2}n\Big)^{\gamma(r,d)}\Big\}.
\end{align*}
Since $\delta_n^{2d/(2r+d)}(\sigma^2/n)^{\gamma(r,d)}=\rho^2(\sigma^2/(n\rho^2))^{\gamma(s,d)}$, the minimum equals $\rho^2(\sigma^2/(n\rho^2))^{\gamma(s,d)}$ times $\min\{1,(\delta/\delta_n)^{2d/(2r+d)}\}$, which proves \cref{thm:universal}.
\end{proof}

\subsection{Proof of \texorpdfstring{\cref{cor:naive}}{the corollary on bounds vanishing with the residual}}
\begin{proof}
Suppose that such an estimator exists, and fix $\delta\in(0,1]$.
Since $s>r$, $\delta_n\to0$, so for $n$ large enough $\delta\geq\delta_n$ and the sample-size conditions of \cref{thm:universal} hold.
That theorem and the assumed bound then give $c\rho^2(\sigma^2/\rho^2)^{\gamma(s,d)}\leq Ch(\delta)+Cn^{\gamma(s,d)-\gamma(s,d-1)}$.
Letting $n\to\infty$ and using $\gamma(s,d)<\gamma(s,d-1)$ gives $c\rho^2(\sigma^2/\rho^2)^{\gamma(s,d)}\leq Ch(\delta)$, which fails for $\delta$ small enough.
Hence no estimator satisfies the bound for all $n\geq1$ and $\delta\in(0,1]$.
\end{proof}

\subsection{Localization and separated frequencies}
\begin{lem}[Localization and periodization]
\label[lem]{lem:fourier}
Let $s\geq0$, $q\in\mathbb N^\ast$ and $L>0$, and let $w,\chi\in C_c^\infty(\R^q)$ with $\mathrm{supp}\,\chi\subset(-2L,2L)^q$.
\begin{itemize}
\item[(i)] If $F\in H^s(\R^q)$, the $4L$-periodization $g$ of $\chi F$ satisfies $\|g\|_{H^s(\T^q_L)}\leq C_\chi\|F\|_{H^s(\R^q)}$.
\item[(ii)] If $\Lambda\subset\R^q$ is $\tau$-separated and $c\in\ell^2(\Lambda)$, then $\|w\sum_\nu c_\nu e^{i\langle\nu,\cdot\rangle}\|^2_{H^s(\R^q)}\leq C_{w,\tau}\sum_\nu(1+|\nu|^2)^s|c_\nu|^2$.
In particular, $\|g\|_{H^s(U)}\leq\|wg\|_{H^s(\R^q)}\leq C_w\|g\|_{H^s(\T^q_L)}$ for every $g\in H^s(\T^q_L)$ and every open $U\subseteq\{w=1\}$.
\item[(iii)] If $\mathbf M\in\R^{q\times q}$ is invertible, then $\|(wF)\circ\mathbf M\|_{H^s(\R^q)}\leq C_{w,\mathbf M}\|F\|_{H^s(\R^q)}$ for $F\in H^s(\R^q)$.
\end{itemize}
\end{lem}
\begin{proof}
All three bounds follow from Peetre's inequality $(1+|a|^2)^s\leq2^s(1+|a-b|^2)^s(1+|b|^2)^s$ and the rapid decay of $\hat w$ and $\hat\chi$.
For (i), $\hat g_k=(4L)^{-q}\widehat{\chi F}(\pi k/(2L))$ and $\widehat{\chi F}=(2\pi)^{-q}\hat\chi\ast\hat F$; the Cauchy--Schwarz inequality gives $|\widehat{\chi F}(\pi k/(2L))|^2\leq C_\chi\int|\hat\chi(\pi k/(2L)-\zeta)||\hat F(\zeta)|^2d\zeta$, and Peetre's inequality with $\sup_\zeta\sum_k|\hat\chi(\pi k/(2L)-\zeta)|(1+|\pi k/(2L)-\zeta|^2)^s<\infty$ gives $\|g\|_{H^s(\T^q_L)}\leq C_\chi\|F\|_{(s)}$.
For (ii) with $\Lambda$ finite, the Fourier transform of the sum is $\sum_\nu c_\nu\hat w(\zeta-\nu)$, and the Cauchy--Schwarz inequality bounds its square by $(\sum_\nu|\hat w(\zeta-\nu)|)\sum_\nu|c_\nu|^2|\hat w(\zeta-\nu)|$; separation bounds the first factor uniformly, and Peetre's inequality gives the claim; the general case follows by completion, and the periodic case by taking $\nu=\pi k/(2L)$ and $c_\nu=\hat g_k$.
For (iii), the same convolution estimate gives $\|wF\|_{(s)}\leq C_w\|F\|_{(s)}$, and $\widehat{(wF)\circ\mathbf M}(\zeta)=|\det\mathbf M|^{-1}\widehat{wF}(\mathbf M^{-\top}\zeta)$ gives $\|(wF)\circ\mathbf M\|_{(s)}\leq C_{\mathbf M}\|wF\|_{(s)}$.
The equivalence of $\|\cdot\|_{(s)}$ and $\|\cdot\|_{H^s(\R^q)}$ concludes the proof of (i)--(iii).
\end{proof}

\begin{lem}[Separated frequencies]
\label[lem]{lem:separated}
Let $w\in C_c^\infty(\R^q)$ with $w\geq0$ and $w\not\equiv0$.
There is $\tau_w\geq1$ such that, for every finite $\tau_w$-separated $\Lambda\subset\R^q$ and $c\in\mathbb C^\Lambda$, $\frac12\hat w(0)\|c\|_2^2\leq\int w|\sum_{\nu\in\Lambda}c_\nu e^{i\langle\nu,\cdot\rangle}|^2\leq\frac32\hat w(0)\|c\|_2^2$.
\end{lem}
\begin{proof}
The integral equals $\hat w(0)\|c\|_2^2+\sum_{\nu\ne\nu'}c_\nu\bar c_{\nu'}\hat w(\nu'-\nu)$.
By Schur's test, the off-diagonal sum is at most $\|c\|_2^2\sup_\nu\sum_{\nu'\ne\nu}|\hat w(\nu'-\nu)|$.
For $\tau\geq1$, at most $(5k)^q$ points of a $\tau$-separated set satisfy $k\tau\leq|\nu'-\nu|<(k+1)\tau$, and $|\hat w(\zeta)|\leq C_w(1+|\zeta|)^{-q-2}$, so this supremum is at most $5^qC_w\sum_{k\geq1}k^q(1+k\tau)^{-q-2}\leq5^q\frac{\pi^2}6C_w\tau^{-2}$.
Choosing $\tau_w$ such that this is at most $\hat w(0)/2$ proves both bounds.
\end{proof}

\subsection{Exact elliptic solutions decaying away from a hyperplane}
For $\xi\in\mathbb C^d$, write $\langle\xi,x\rangle=\sum_i\xi_ix_i$; then $\D(e^{\langle\xi,\cdot\rangle})=P(\xi)e^{\langle\xi,\cdot\rangle}$, and $\D(\mathrm{Re}\,g)=\mathrm{Re}\,\D(g)$ because the coefficients are real.
For $y\in\mathbb C$ and $\eta\in\R^{d-1}$, $P(y,i\eta)$ denotes $P$ evaluated at $(y,i\eta)\in\mathbb C^d$, and $P_j$ denotes the homogeneous part of degree $j$ of $P$.

\begin{lem}[Decaying exact solutions]
\label[lem]{lem:decay}
Let $\D$ be elliptic of order $r$ with constant real coefficients.
There are $c_+,C_z,K_2>0$ such that, for every $\eta\in\R^{d-1}$ with $|\eta|\geq K_2$, there is $z(\eta)\in\mathbb C$ with $P(z(\eta),i\eta)=0$, $\mathrm{Re}\,z(\eta)\geq c_+|\eta|$ and $|z(\eta)|\leq C_z|\eta|$.
In particular, $\psi_\eta(x)=\mathrm{Re}\,e^{z(\eta)x_1+i\langle\eta,x'\rangle}$, where $x=(x_1,x')$, satisfies $\D(\psi_\eta)=0$.
\end{lem}
\begin{proof}
For $e'\in\R^{d-1}$ with $|e'|=1$ and $t\in[0,1]$, let $p_{e',t}(y)=\sum_{j=0}^rt^{r-j}P_j(y,ie')$, a polynomial in $y$ with leading coefficient $P_r(1,0)\neq0$ and coefficients continuous in $(e',t)$.
Its roots are not imaginary at $t=0$, since $P_r(ia,ie')=i^rP_r(a,e')\neq0$ for $a\in\R$ by ellipticity.
Since $r$ is even and the coefficients are real, $\overline{P_r(y,ie')}=P_r(\bar y,-ie')=P_r(-\bar y,ie')$, so the roots of $p_{e',0}$ are symmetric with respect to the imaginary axis, and half of them have positive real part.
By Rouch\'e's theorem, the largest real part of the roots of $p_{e',t}$ is continuous in $(e',t)$; let $2c_+>0$ be its minimum over the compact set $\{t=0,|e'|=1\}$.
Uniform continuity gives $t_\star\in(0,1]$ such that this largest real part is at least $c_+$ for $t\leq t_\star$, and Cauchy's root bound gives $C_z$ bounding the modulus of every root for every $(e',t)$.
Let $K_2=1/t_\star$, and, for $|\eta|\geq K_2$, let $y$ be a root of $p_{e',t}$ with $e'=\eta/|\eta|$, $t=1/|\eta|$ and $\mathrm{Re}\,y\geq c_+$.
Since $p_{e',t}(y)=t^rP(y/t,i\eta)$, the number $z(\eta)=|\eta|y$ has the three properties, and $\D(\psi_\eta)=\mathrm{Re}(P(z(\eta),i\eta)e^{z(\eta)x_1+i\langle\eta,x'\rangle})=0$.
\end{proof}

\subsection{Proof of \texorpdfstring{\cref{prop:ellexact}}{the lower bound for exact elliptic solutions}}
\begin{proof}
A rotation followed by a translation of $\R^d$ preserves uniform design, changes $H^s$ norms by at most a constant factor, and maps $\D$ to an elliptic operator with constant real coefficients.
Applying one that sends $\mathbf e$ to $(1,0,\ldots,0)$ and a point of $V$ to the origin, with $\mathbf e$ and $V$ as in the definition of a flat face, we may therefore assume that $\Omega\subseteq(-\infty,0)\times[-L',L']^{d-1}$ with $L'=2\sqrt d\,L\geq\mathrm{diam}\,\Omega$, and that $(-a_0,0)\times Q'\subseteq\Omega$ for an open cube $Q'\subset\R^{d-1}$.
Let $w_0\in C_c^\infty(Q')$ with $0\leq w_0\leq1$ and $w_0\not\equiv0$, let $\tau=\tau_{w_0}$ be given by \cref{lem:separated} with $q=d-1$, and, for $K\geq K_2$, let $\Lambda=\{\eta\in\tau\Z^{d-1}:K\leq|\eta|\leq2K,\ \eta_1\geq K/2\}$, with $\psi_\eta$ given by \cref{lem:decay}.
The region defining $\Lambda$ contains a ball of radius $K/4$ and lies in a ball of radius $2K$, so $cK^{d-1}\leq\#\Lambda\leq CK^{d-1}$ for $K\geq C$.
Moreover, $\Lambda\cup(-\Lambda)$ is $\tau$-separated, since $|\eta+\tilde\eta|\geq\eta_1+\tilde\eta_1\geq K$ for $\eta,\tilde\eta\in\Lambda$.

Let $v\in\R^\Lambda$.
For fixed $x_1\in(-a_0,0)$, writing $\psi_\eta=\frac12(e^{z(\eta)x_1+i\langle\eta,x'\rangle}+e^{\overline{z(\eta)}x_1-i\langle\eta,x'\rangle})$, \cref{lem:separated} gives $\int w_0|\sum_\eta v_\eta\psi_\eta(x_1,x')|^2dx'\geq\frac14\hat w_0(0)\sum_\eta v_\eta^2e^{2\mathrm{Re}\,z(\eta)x_1}$.
Since $c_+K\leq\mathrm{Re}\,z(\eta)\leq2C_zK$, integration over $x_1\in(-a_0,0)$ gives $\|\sum_\eta v_\eta\psi_\eta\|^2_{L^2(\Omega)}\geq cK^{-1}\|v\|_2^2$ for $K\geq C$.
Conversely, $|\psi_\eta(x)|\leq e^{\mathrm{Re}\,z(\eta)x_1}$ and $\Omega\subseteq(-\infty,0)\times[-L',L']^{d-1}$ give $\|\psi_\eta\|^2_{L^2(\Omega)}\leq(2L')^{d-1}/(2c_+K)$.
Thus the assumptions of \cref{lem:testing} hold with $b=c/K$ and $B_0=C/K$.

Let $\chi\in C^\infty(\R)$ equal $1$ on $(-\infty,0]$ and $0$ on $[1,\infty)$, and let $w_1\in C_c^\infty(\R^{d-1})$ equal $1$ on $[-L',L']^{d-1}$.
Set $\varphi_\eta(x_1)=\chi(Kx_1)e^{z(\eta)x_1}$ and $G(x)=w_1(x')\sum_\eta v_\eta\varphi_\eta(x_1)e^{i\langle\eta,x'\rangle}$, so that $\sum_\eta v_\eta\psi_\eta=\mathrm{Re}\,G=\frac12(G+\overline G)$ on $\Omega$.
We have $\varphi_\eta(x_1)=\Phi_{z(\eta)/K}(Kx_1)$ with $\Phi_\xi(y)=\chi(y)e^{\xi y}$, and $z(\eta)/K$ lies in the compact set $\{\xi\in\mathbb C:\mathrm{Re}\,\xi\geq c_+,|\xi|\leq2C_z\}$, on which $\xi\mapsto\|\partial^j\Phi_\xi\|_{L^2(\R)}$ is continuous; hence $\|\partial^j\varphi_\eta\|^2_{L^2(\R)}\leq C_jK^{2j-1}$ for every integer $j\geq0$.
For $\alpha=(\alpha_1,\alpha')$, the Leibniz rule writes $\partial^\alpha G$ as a finite sum of terms $\partial^{\alpha'-\beta}w_1(x')\sum_\eta v_\eta\partial^{\alpha_1}\varphi_\eta(x_1)(i\eta)^\beta e^{i\langle\eta,x'\rangle}$ with $\beta\leq\alpha'$.
For fixed $x_1$, \cref{lem:fourier}(ii) at order $0$ in $x'$ bounds the squared $L^2(\R^{d-1})$ norm of such a term by $C\sum_\eta v_\eta^2|\partial^{\alpha_1}\varphi_\eta(x_1)|^2|\eta|^{2|\beta|}$, and integration in $x_1$ gives $\|\partial^\alpha G\|^2_{L^2(\R^d)}\leq CK^{2|\alpha|-1}\|v\|_2^2$.
Hence $\|G\|^2_{(j)}\leq CK^{2j-1}\|v\|_2^2$ for $j\in\{0,\lceil s\rceil\}$, and H\"older's inequality in the Fourier integral, $\|G\|^2_{(s)}\leq\|G\|_{(0)}^{2(1-s/\lceil s\rceil)}\|G\|_{(\lceil s\rceil)}^{2s/\lceil s\rceil}$, gives $\|\sum_\eta v_\eta\psi_\eta\|^2_{H^s(\Omega)}\leq\|G\|^2_{H^s(\R^d)}\leq C\|G\|^2_{(s)}\leq CK^{2s-1}\|v\|_2^2$.

Every $\psi_\eta$ solves $\D(\psi_\eta)=0$, so \cref{lem:testing}(ii) applies with $A_s=CK^{2s-1}$, $A_\D=0$ and $M=\#\Lambda\leq CK^{d-1}$ as soon as $C\sigma^2K^{2s+d-1}\leq n\rho^2$.
Take $K=c_\star(n\rho^2/\sigma^2)^{1/(2s+d-1)}$ with $c_\star$ small enough; $n\geq C_0\sigma^2/\rho^2$ with $C_0$ large makes $K$ large enough for the construction.
Since $b/B_0\geq c$ and $\#\Lambda\geq cK^{d-1}$, \cref{lem:testing}(i) gives $\M_n(\F(\rho,0))\geq c\sigma^2K^{d-1}/n=c\rho^2(\sigma^2/(n\rho^2))^{\gamma(s,d-1)}$.
As the rotation and translation only change constants, this proves \cref{prop:ellexact}.
\end{proof}

\subsection{Normal form of hyperbolic operators}
For $r=2$, replacing $\D$ by $-\D$ changes neither $\F(\rho,\delta)$, nor~\eqref{eq:estimator}, nor $\mathcal N(\lambda,\mu)$, so we assume that $(a_{ij})$ has exactly one positive eigenvalue.
The canonical operators are $\Dc=\partial_1$ ($r=1$) and $\Dc=\partial_1^2-\sum_{i\geq2}\partial_i^2$ ($r=2$).

\begin{lem}[Normal form]
\label[lem]{lem:normal}
Let $\D$ be hyperbolic of order $r\in\{1,2\}$ with constant real coefficients.
There are $\beta\in\R^d$, an invertible $\mathbf M\in\R^{d\times d}$ and $c_0\in\R$, with $c_0=0$ if $r=1$, such that, for every open $V$ and $f\in H^r(V)$, $F(y)=e^{-\langle\beta,\mathbf My\rangle}f(\mathbf My)$ satisfies $\D(f)(\mathbf My)=e^{\langle\beta,\mathbf My\rangle}(\Dc F+c_0F)(y)$ for almost every $y\in\mathbf M^{-1}V$.
The first column of $\mathbf M$ is $b$ if $r=1$, and $(a_{ij})=\mathbf MJ_0\mathbf M^\top$ with $J_0=\mathrm{diag}(1,-1,\ldots,-1)$ if $r=2$.
In particular, for $u=(u_1,u')\in\R\times\R^{d-1}$, $x\mapsto e^{\langle\beta,x\rangle}\cos\langle\mathbf M^{-\top}u,x\rangle$ solves $\D(f)=0$ if $u_1=0$ ($r=1$) or $u_1^2=|u'|^2+c_0$ ($r=2$).
\end{lem}
\begin{proof}
Write $f=e^{\langle\beta,\cdot\rangle}g$, so that $F=g\circ\mathbf M$.
For $r=1$, take $\beta=-b_0b/|b|^2$ and any invertible $\mathbf M$ with first column $b$: then $\D(f)=e^{\langle\beta,\cdot\rangle}\langle b,\nabla g\rangle$ and $\partial_1F(y)=\langle b,\nabla g(\mathbf My)\rangle$.
For $r=2$, let $\mathbf A=(a_{ij})$; the spectral theorem gives $\mathbf A=\mathbf Q\Lambda\mathbf Q^\top$ with $\mathbf Q$ orthogonal and $\Lambda=\mathrm{diag}(\lambda_1,\ldots,\lambda_d)$, $\lambda_1>0>\lambda_2,\ldots,\lambda_d$, so that $\mathbf M=\mathbf Q|\Lambda|^{1/2}$ satisfies $\mathbf A=\mathbf MJ_0\mathbf M^\top$.
With $\beta=-\frac12\mathbf A^{-1}b$ and $c_0=b_0-\frac14\langle b,\mathbf A^{-1}b\rangle$, the Leibniz rule gives $\D(f)=e^{\langle\beta,\cdot\rangle}(\mathrm{tr}(\mathbf A\nabla^2g)+c_0g)$, and $\nabla^2F(y)=\mathbf M^\top\nabla^2g(\mathbf My)\mathbf M$ gives $\Dc F(y)=\mathrm{tr}(J_0\nabla^2F(y))=\mathrm{tr}(\mathbf A\nabla^2g(\mathbf My))$.
These identities hold for weak derivatives.
Finally, $F(y)=\cos\langle u,y\rangle$ gives $(\Dc+c_0)F=-u_1\sin\langle u,\cdot\rangle$ if $r=1$ and $(|u'|^2+c_0-u_1^2)\cos\langle u,\cdot\rangle$ if $r=2$, which vanish under the stated conditions.
\end{proof}

\subsection{Plane-wave testing families for hyperbolic operators}
Let $\beta$, $\mathbf M$ and $c_0$ be given by \cref{lem:normal}, write $u=(u_1,u')\in\R\times\R^{d-1}$, and let $p(u)=iu_1$ if $r=1$ and $p(u)=|u'|^2+c_0-u_1^2$ if $r=2$.
For $\zeta=\mathbf M^{-\top}u$ and $f(x)=e^{\langle\beta,x\rangle+i\langle\zeta,x\rangle}$, the function $F$ of \cref{lem:normal} (which applies to complex-valued functions by linearity) is $F(y)=e^{i\langle u,y\rangle}$, and $\Dc F+c_0F=p(u)F$; hence $\D(f)=p(u)f$ on $\Omega$.

\begin{lem}[Plane waves]
\label[lem]{lem:planewaves}
Let $\D$ be hyperbolic of order $r\in\{1,2\}$ with constant real coefficients, and let $C_1>0$.
There are $\tau\geq1$ and $b,B_0,C>0$, with $C$ depending also on $C_1$, such that the following holds for every $K\geq1$ and every finite $U\subset\R^d$ with $|u|\leq C_1K$ for $u\in U$ and $|u-\tilde u|\geq\tau$ for distinct $u,\tilde u\in U$ and $|u+\tilde u|\geq\tau$ for all $u,\tilde u\in U$.
The functions $\psi_u(x)=e^{\langle\beta,x\rangle}\cos\langle\mathbf M^{-\top}u,x\rangle$ satisfy $\|\psi_u\|^2_{L^2(\PP_X)}\leq B_0$ and, for every $v\in\R^U$,
\begin{gather*}
\Big\|\sum_uv_u\psi_u\Big\|^2_{L^2(\PP_X)}\geq b\|v\|_2^2,\qquad
\Big\|\sum_uv_u\psi_u\Big\|^2_{H^s(\Omega)}\leq CK^{2s}\|v\|_2^2,\\
\Big\|\D\Big(\sum_uv_u\psi_u\Big)\Big\|^2_{L^2(\Omega)}\leq C\max_{u\in U}|p(u)|^2\|v\|_2^2.
\end{gather*}
\end{lem}
\begin{proof}
Let $w_0\in C_c^\infty(\Omega)$ with $0\leq w_0\leq1$ and $w_0\not\equiv0$, let $w=w_0e^{2\langle\beta,\cdot\rangle}$, and let $\tau=\tau_w\max(1,\|\mathbf M\|_{\mathrm{op}})$, where $\tau_w$ is given by \cref{lem:separated}.
Write $\zeta_u=\mathbf M^{-\top}u$.
The set $\Lambda=\{\pm\zeta_u:u\in U\}$ is $\tau_w$-separated, because $|\mathbf M^{-\top}(u-\tilde u)|\geq|u-\tilde u|/\|\mathbf M\|_{\mathrm{op}}$, and $|\zeta|\leq\|\mathbf M^{-1}\|_{\mathrm{op}}C_1K$ for $\zeta\in\Lambda$.
We have $\psi_u=\frac12e^{\langle\beta,\cdot\rangle}(e^{i\langle\zeta_u,\cdot\rangle}+e^{-i\langle\zeta_u,\cdot\rangle})$ and, by the identity above and the reality of the coefficients, $\D(\psi_u)=\frac12e^{\langle\beta,\cdot\rangle}(p(u)e^{i\langle\zeta_u,\cdot\rangle}+\overline{p(u)}e^{-i\langle\zeta_u,\cdot\rangle})$ on $\Omega$.
Since $0\leq w_0\leq1$, \cref{lem:separated} gives $\|\sum_uv_u\psi_u\|^2_{L^2(\Omega)}\geq\int w|\sum_u\frac{v_u}2(e^{i\langle\zeta_u,\cdot\rangle}+e^{-i\langle\zeta_u,\cdot\rangle})|^2\geq\frac14\hat w(0)\|v\|_2^2$, so $b=\hat w(0)/(4|\Omega|)$ works, and $B_0=e^{2\sqrt dL|\beta|}$ bounds $|\psi_u|^2$ on $\Omega\subseteq[-L,L]^d$.
Let $\chi\in C_c^\infty(\R^d)$ equal $1$ on $\Omega$.
On $\Omega$, $\sum_uv_u\psi_u$ and $\D(\sum_uv_u\psi_u)$ coincide with $\chi e^{\langle\beta,\cdot\rangle}$ times exponential sums over $\Lambda$ whose coefficients have moduli $|v_u|/2$ and $|p(u)v_u|/2$, respectively.
Thus \cref{lem:fourier}(ii) at orders $s$ and $0$ gives the last two bounds, since $(1+|\zeta|^2)^s\leq CK^{2s}$ for $\zeta\in\Lambda$.
\end{proof}

For $K\geq1$, let $\mathcal U'_K=\{u'\in\tau\Z^{d-1}:K\leq|u'|\leq2K,\ u'_1\geq K/2\}$, where $u'_1$ is the first coordinate of $u'$ and $\tau$ is given by \cref{lem:planewaves}.
As for $\Lambda$ in the proof of \cref{prop:ellexact}, $cK^{d-1}\leq\#\mathcal U'_K\leq CK^{d-1}$ for $K\geq C$, and $|u'+\tilde u'|\geq u'_1+\tilde u'_1\geq K$ for $u',\tilde u'\in\mathcal U'_K$.

\begin{lem}[Exact solutions]
\label[lem]{lem:exact}
Let $\D$ be hyperbolic of order $r\in\{1,2\}$ with constant real coefficients, and let $s\geq r$.
There are $c,C_0>0$ such that, for $n\geq C_0\sigma^2/\rho^2$, $\M_n(\F(\rho,0))\geq c\rho^2(\sigma^2/(n\rho^2))^{\gamma(s,d-1)}$.
\end{lem}
\begin{proof}
For $K\geq C$, let $U=\{(u_1,u'):u'\in\mathcal U'_K\}$, with $u_1=0$ if $r=1$ and $u_1=(|u'|^2+c_0)^{1/2}$ if $r=2$, which is real since $|u'|^2\geq K^2\geq|c_0|$.
Then $p(u)=0$ and $|u|\leq5K$ for $u\in U$.
For $u,\tilde u\in U$, $|u+\tilde u|\geq|u'+\tilde u'|\geq K\geq\tau$, and $|u-\tilde u|\geq|u'-\tilde u'|\geq\tau$ if $u\neq\tilde u$.
Thus \cref{lem:planewaves} and \cref{lem:testing}(ii) apply with $M=\#U\leq CK^{d-1}$, $A_s=CK^{2s}$ and $A_\D=0$ as soon as $C\sigma^2K^{2s+d-1}\leq n\rho^2$.
Take $K=c_\star(n\rho^2/\sigma^2)^{1/(2s+d-1)}$ with $c_\star$ small enough, and $C_0$ large enough that $K\geq C$.
Since $b/B_0\geq c$ and $\#U\geq cK^{d-1}$, \cref{lem:testing}(i) gives $\M_n(\F(\rho,0))\geq c\sigma^2K^{d-1}/n=c\rho^2(\sigma^2/(n\rho^2))^{\gamma(s,d-1)}$.
\end{proof}

\begin{lem}[Frequencies near the characteristic set]
\label[lem]{lem:nearcone}
Let $\D$ be hyperbolic of order $r\in\{1,2\}$ with constant real coefficients, let $s\geq r$, and let $K_\star=(n\rho^3/(\sigma^2\delta))^{1/(3s+d-r)}$, with $K_\star=\infty$ if $\delta=0$.
There are $c,c_1,c_2,C_0>0$ such that, if $n\geq C_0\sigma^2/\rho^2$ and
\[
\max\Big\{K_\star,c_2\Big(\frac{n\rho^2}{\sigma^2}\Big)^{1/(2s+d)}\Big\}\leq K\leq c_1\Big(\frac{n\rho^2}{\sigma^2}\Big)^{1/(2s+d-1)},
\]
then $\M_n(\F(\rho,\delta))\geq c\rho^2K^{-2s}$.
\end{lem}
\begin{proof}
Let $c_\star\in(0,1]$ be a constant fixed below, let $c_1=(c_\star/(2\tau))^{1/(2s+d-1)}$ and $c_2=(2c_\star)^{1/(2s+d)}$, and set $h=c_\star(n\rho^2/\sigma^2)K^{-(2s+d-r)}$.
The two bounds on $K$ then give $2\tau K^{r-1}\leq h\leq K^r/2$, and $C_0$ large enough gives $K\geq C$.
Let $U$ be the set of $u=(u_1,u')$ with $u'\in\mathcal U'_K$, $t\in\tau\Z\cap[hK^{1-r},2hK^{1-r}]$, and $u_1=t$ if $r=1$, $u_1=|u'|+t$ if $r=2$.
The interval for $t$ has length $hK^{1-r}\geq2\tau$, so it contains between $hK^{1-r}/(2\tau)$ and $2hK^{1-r}/\tau$ points, and $chK^{d-r}\leq\#U\leq ChK^{d-r}$.
Since $t\leq2hK^{1-r}\leq K$, $|u|\leq5K$ for $u\in U$.
Distinct pairs $(u',t)$ give $|u-\tilde u|\geq\tau$, and $|u+\tilde u|\geq|u'+\tilde u'|\geq K\geq\tau$.
Moreover, $|p(u)|=t\leq2h$ if $r=1$, and $|p(u)|=|c_0-2|u'|t-t^2|\leq|c_0|+8h+4h^2K^{-2}\leq|c_0|+10h\leq Ch$ if $r=2$, since then $h\geq2\tau K\geq1$.
Hence \cref{lem:planewaves} and \cref{lem:testing}(ii) apply with $M=\#U$, $A_s=CK^{2s}$ and $A_\D=Ch^2$ as soon as $C\sigma^2hK^{2s+d-r}\leq n\rho^2$ and $C\sigma^2h^3K^{d-r}\leq n\delta^2$.
The first condition reads $Cc_\star\leq1$.
Since $K\geq K_\star$, $C\sigma^2h^3K^{d-r}=Cc_\star^3(n^3\rho^6/\sigma^4)K^{-2(3s+d-r)}\leq Cc_\star^3n\delta^2$, so the second condition follows from $Cc_\star^3\leq1$.
Choosing $c_\star$ small enough, \cref{lem:testing}(i) gives $\M_n(\F(\rho,\delta))\geq c\sigma^2hK^{d-r}/n=cc_\star\rho^2K^{-2s}$.
\end{proof}

\subsection{Proof of the lower bound of \texorpdfstring{\cref{thm:hyperbolic}}{the hyperbolic theorem}}
\begin{proof}
Enlarge $C_0$ so that \cref{lem:exact,lem:nearcone} apply and $(n\rho^2/\sigma^2)^{1/((2s+d)(2s+d-1))}\geq c_2/c_1$, that is, $c_2(n\rho^2/\sigma^2)^{1/(2s+d)}\leq c_1(n\rho^2/\sigma^2)^{1/(2s+d-1)}$.
Since $\F(\rho,0)\subseteq\F(\rho,\delta)$, \cref{lem:exact} gives $\M_n(\F(\rho,\delta))\geq c\rho^2(\sigma^2/(n\rho^2))^{\gamma(s,d-1)}$.
Let $K_\star$ be as in \cref{lem:nearcone}, so that $\rho^2K_\star^{-2s}=\rho^2(\sigma^2\delta/(n\rho^3))^{2s/(3s+d-r)}$, and let $K$ be the maximum of $K_\star$ and $c_2(n\rho^2/\sigma^2)^{1/(2s+d)}$.
If $K\leq c_1(n\rho^2/\sigma^2)^{1/(2s+d-1)}$, \cref{lem:nearcone} gives
\[
\M_n(\F(\rho,\delta))\geq c\rho^2K^{-2s}\geq c\rho^2\min\Big\{\Big(\frac{\sigma^2\delta}{n\rho^3}\Big)^{2s/(3s+d-r)},\Big(\frac{\sigma^2}{n\rho^2}\Big)^{\gamma(s,d)}\Big\}.
\]
Otherwise, $K_\star>c_1(n\rho^2/\sigma^2)^{1/(2s+d-1)}$, so $\rho^2(\sigma^2\delta/(n\rho^3))^{2s/(3s+d-r)}=\rho^2K_\star^{-2s}$ is at most $c_1^{-2s}\rho^2(\sigma^2/(n\rho^2))^{\gamma(s,d-1)}$, and the exact-solution bound already controls this minimum.
In both cases, $\M_n(\F(\rho,\delta))$ is at least $c$ times the maximum, hence $c/2$ times the sum, of the exact-solution term and the minimum.
This proves the lower bound of \cref{thm:hyperbolic} for all $s\geq r$, $\delta\geq0$ and $n\geq C_0\sigma^2/\rho^2$.
\end{proof}

\section{Effective-dimension bounds}
\label{app:count}
All linear spaces are complex.
We write $Q_0=(-1,1)^d$, and $Q^\ast$ for the cube with the same center as a cube $Q$ and $6/5$ times its side, so that $Q_0^\ast=(-6/5,6/5)^d$.

\subsection{Counting by approximation}
\begin{lem}[Counting by approximation]
\label[lem]{lem:counting}
Assume (A1), and let $\D$ be linear with constant coefficients, $\lambda>0$ and $\mu\geq0$.
\begin{itemize}
\item[(i)] For Hermitian matrices $0\preceq A\preceq A'$ and $B\succ0$, $\mathrm{tr}((A+B)^{-1}A)\leq\mathrm{tr}((A'+B)^{-1}A')$.
\item[(ii)] Let $n:(0,1]\to[0,\infty)$ be measurable, and assume that, for every $\tau\in(0,1]$, there is a linear space $W_\tau\subseteq L^2(\Omega)$ of dimension at most $n(\tau)$ such that $\mathrm{dist}_{L^2(\Omega)}(f_\theta,W_\tau)^2\leq1/(3\kappa)$ for every $\theta$ with $\tau\lambda\|S\theta\|_2^2+\tau\mu\|\D(f_\theta)\|^2_{L^2(\Omega)}\leq1$ and $\kappa\|f_\theta\|^2_{L^2(\Omega)}\leq1$.
Then $\mathcal N(\lambda,\mu)\leq\int_0^1n(\tau)\,d\tau$.
In particular, if $n(\tau)\leq\sum_iA_i\tau^{-a_i}(1+\log(1/\tau))$ for finitely many $A_i\geq0$ and $a_i\in[0,1)$, then $\mathcal N(\lambda,\mu)\leq\sum_i2A_i(1-a_i)^{-2}$.
\end{itemize}
\end{lem}
\begin{proof}
(i) follows from $\mathrm{tr}((A+B)^{-1}A)=\mathrm{tr}(I)-\mathrm{tr}(B^{1/2}(A+B)^{-1}B^{1/2})$ and the monotonicity of the inverse for the Loewner order.
For (ii), let $B=\lambda S^2+\mu D^\ast GD\succ0$.
Since $\theta^\ast\Sigma\theta=\EE|f_\theta(X)|^2\leq\kappa\|f_\theta\|^2_{L^2(\Omega)}=\kappa\theta^\ast G\theta$, (i) gives $\mathcal N(\lambda,\mu)\leq\mathrm{tr}((\kappa G+B)^{-1}\kappa G)=\sum_i\nu_i/(1+\nu_i)$, where the $\nu_i$ are the eigenvalues of $\kappa B^{-1/2}GB^{-1/2}$.
As $\nu/(1+\nu)\leq\min(1,2\nu)=\int_0^1\mathbf 1_{\{2\nu>\tau\}}d\tau$, it suffices to show that $\#\{i:2\nu_i>\tau\}\leq n(\tau)$.
By the Courant--Fischer theorem, this number is the largest dimension of a subspace $V$ on which $2\kappa\theta^\ast G\theta>\tau\theta^\ast B\theta$ for every $\theta\neq0$.
If $\dim V>n(\tau)\geq\dim W_\tau$, some nonzero $\theta\in V$ has $f_\theta\perp W_\tau$; normalized by $\kappa\|f_\theta\|^2_{L^2(\Omega)}+\tau\theta^\ast B\theta=1$, it satisfies both constraints and $\kappa\|f_\theta\|^2_{L^2(\Omega)}>1/3$, whereas $\kappa\|f_\theta\|^2_{L^2(\Omega)}=\kappa\,\mathrm{dist}_{L^2(\Omega)}(f_\theta,W_\tau)^2\leq1/3$, a contradiction.
Finally, $\int_0^1\tau^{-a}(1+\log(1/\tau))\,d\tau=(1-a)^{-1}+(1-a)^{-2}\leq2(1-a)^{-2}$ for $a\in[0,1)$.
\end{proof}

\subsection{Polynomial approximation}
\begin{lem}[Moment projection]
\label[lem]{lem:poly}
Fix $w\in C_c^\infty(Q_0^\ast)$ with $\int w=1$.
For every integer $j\geq1$, there is a unique map $\Pi_j:L^2(Q_0^\ast)\to\Pol_j$ such that $\int(\Pi_jf)\,\partial^\alpha w=\int f\,\partial^\alpha w$ for $|\alpha|<j$; it is linear and bounded, and $\Pi_jp=p$ for $p\in\Pol_j$.
Let $s\geq1$ be real and $\bar s=\lceil s\rceil$.
\begin{itemize}
\item[(a)] For $f\in H^s(Q_0^\ast)$, $\|f-\Pi_{\bar s}f\|_{H^s(Q_0^\ast)}\leq C|f|_{s,Q_0^\ast}$; for every cube $P$ of side $g$ and $f\in H^s(P)$, $\inf_{p\in\Pol_{\bar s}}\|f-p\|_{L^2(P)}\leq Cg^s|f|_{s,P}$.
\item[(b)] For $f\in H^s(Q_0^\ast)$ and every homogeneous operator $\D_0$ of order $r\leq s$ with constant coefficients, $\|\D_0(f-\Pi_{\bar s}f)\|_{L^2(Q_0^\ast)}\leq C\|\D_0f\|_{L^2(Q_0^\ast)}$.
\end{itemize}
\end{lem}
\begin{proof}
If $p\in\Pol_j$ has vanishing moments, then $\int(\partial^\alpha p)w=(-1)^{|\alpha|}\int p\,\partial^\alpha w=0$ for $|\alpha|<j$, and taking $|\alpha|=\deg p,\deg p-1,\ldots$ shows that $p=0$; as $\dim\Pol_j=\#\{\alpha:|\alpha|<j\}$, the moment map from $\Pol_j$ to $\mathbb C^{\{\alpha:|\alpha|<j\}}$ is a bijection, which defines $\Pi_j$, bounded since the moments are continuous on $L^2$, with $\Pi_jp=p$ by uniqueness.
(a) On a cube $U$, which is star-shaped with respect to a ball, the Bramble--Hilbert estimates of \citet[Theorems~3.2 and~6.1]{dupont1980polynomial}, with exponent $2$, give $q\in\Pol_{\bar s}$ with $\|f-q\|_{H^s(U)}\leq C_U|f|_{s,U}$ (Theorem~3.2 for integer $s$, where $|f-q|_{s,U}=|f|_{s,U}$ because $\deg q<s$; Theorem~6.1 otherwise).
For $U=Q_0^\ast$, since $\Pi_{\bar s}q=q$ and all norms are equivalent on $\Pol_{\bar s}$, $\|f-\Pi_{\bar s}f\|_{H^s}\leq\|f-q\|_{H^s}+\|\Pi_{\bar s}(f-q)\|_{H^s}\leq\|f-q\|_{H^s}+C\|f-q\|_{L^2}\leq C|f|_{s,Q_0^\ast}$, which is the first claim.
For $U=(0,1)^d$, dilation gives the second claim on $P$, since $\|\cdot\|^2_{L^2}$ and $|\cdot|^2_s$ scale like $g^d$ and $g^{d-2s}$.
(b) If $\bar s\leq r$, then $\D_0\Pi_{\bar s}f=0$.
Otherwise $\D_0\Pi_{\bar s}f\in\Pol_{\bar s-r}$ and, for $|\alpha|<\bar s-r$, integrating by parts twice, with derivatives of $w$ of order less than $\bar s$, gives $\int(\D_0\Pi_{\bar s}f)\partial^\alpha w=(-1)^r\int(\Pi_{\bar s}f)\D_0\partial^\alpha w=(-1)^r\int f\,\D_0\partial^\alpha w=\int(\D_0f)\partial^\alpha w$, so that $\D_0\Pi_{\bar s}f=\Pi_{\bar s-r}\D_0f$ by uniqueness.
In both cases, $\|\D_0(f-\Pi_{\bar s}f)\|_{L^2}\leq(1+\|\Pi_{\bar s-r}\|)\|\D_0f\|_{L^2}$, with $\Pi_0=0$.
\end{proof}

\subsection{Local approximation for hyperbolic operators}
\begin{lem}[Local approximation, hyperbolic case]
\label[lem]{lem:local}
Let $r\in\{1,2\}$, $s\geq r$ and $\varepsilon>0$.
There is $C_\varepsilon$ such that, for every $\lambda\in(0,1]$ and $\mu>0$, there is a linear space $V$ of functions on $Q_0$ with $\dim V\leq C_\varepsilon(\lambda^{-(d-1)/(2s)}+\lambda^{-(d-r)/(2s)}\mu^{-1/2}\ell_{r,d}(\lambda))$ such that $\mathrm{dist}_{L^2(Q_0)}(q,V)^2\leq\varepsilon^2(\lambda|q|^2_{s,Q_0^\ast}+\mu\|\Dc q\|^2_{L^2(Q_0^\ast)})$ for every $q\in H^s(Q_0^\ast)$.
\end{lem}
\begin{proof}
Write $y=(t,y')\in\R\times\R^{d-1}$, $\nabla'$ and $\Delta'$ for the spatial gradient and Laplacian, and $\psi_k(y')=e^{i\pi\langle k,y'\rangle/2}$ and $\omega_k=\pi|k|/2$ for $k\in\Z^{d-1}$.
Let $\bar s=\lceil s\rceil$, and split $(-1,1)$ into $32$ intervals $I_j$ of length $1/16$.
The space $V$ will contain $\Pol_{\bar s}$, which contains every $q$ with $|q|_{s,Q_0^\ast}=0$; by homogeneity, we may assume $\lambda|q|^2_{s,Q_0^\ast}+\mu\|\Dc q\|^2_{L^2(Q_0^\ast)}=1$.

\emph{Decomposition.}
Let $u=q-\Pi_{\bar s}q$, and let $U=Eu$ be its extension (\cref{app:tools}); by \cref{lem:poly}, $\|U\|_{(s)}\leq C\lambda^{-1/2}$ and $\|\Dc u\|_{L^2(Q_0^\ast)}\leq C\mu^{-1/2}$.
Since $(1+|\zeta'|^2)^s+\zeta_1^2(1+|\zeta'|^2)^{s-1}\leq2(1+|\zeta|^2)^s$ for $\zeta=(\zeta_1,\zeta')$, Plancherel's theorem in $t$ gives $\int_\R\phi(t)\,dt\leq C\lambda^{-1}$, where $\phi(t)=\|U(t,\cdot)\|^2_{(s)}+\|\partial_tU(t,\cdot)\|^2_{(s-1)}$ with norms on $\R^{d-1}$.
For $r=1$, let $h=0$.
For $r=2$, $U\in H^1(\R;H^1(\R^{d-1}))$ and $\partial_tU\in H^1(\R;L^2(\R^{d-1}))$, so we may take $t\mapsto(U(t,\cdot),\partial_tU(t,\cdot))$ continuous into $H^1\times L^2$; choose $t_j\in I_j$ with $\phi(t_j)\leq32\int\phi\leq C\lambda^{-1}$, let $\chi_0\in C_c^\infty((-2,2)^{d-1})$ equal $1$ on $(-6/5,6/5)^{d-1}$, and let $u_0,u_1$ be the periodizations of $\chi_0U(t_j,\cdot)$ and $\chi_0\partial_tU(t_j,\cdot)$, so that $\|u_0\|_{H^s(\T^{d-1}_1)}+\|u_1\|_{H^{s-1}(\T^{d-1}_1)}\leq C\lambda^{-1/2}$ by \cref{lem:fourier}(i).
Let $h$ solve $(\partial_t^2-\Delta')h=0$ on $\R\times\T^{d-1}_1$ with $h(t_j)=u_0$ and $\partial_th(t_j)=u_1$: its coefficients are $\hat h_k(t)=\hat u_{0,k}\cos(\omega_k(t-t_j))+\hat u_{1,k}\omega_k^{-1}\sin(\omega_k(t-t_j))$ for $k\neq0$ and $\hat h_0(t)=\hat u_{0,0}+(t-t_j)\hat u_{1,0}$.
Since $\omega_k\geq\pi/2$ for $k\neq0$, $\|h(t)\|_{H^s(\T^{d-1}_1)}\leq C\lambda^{-1/2}$ for $|t-t_j|\leq2$, and $h\in H^2$ on bounded sets because $s\geq2$.
Fix $\chi\in C_c^\infty((-17/16,17/16)^{d-1})$ equal to $1$ on $[-1,1]^{d-1}$, let $w=U-h$, and let $v$ be the periodization of $\chi w$ on $I_j\times\T^{d-1}_1$, so that $q=\Pi_{\bar s}q+h+v$ on $I_j\times(-1,1)^{d-1}$.
By \cref{lem:fourier}(i)--(ii), $\|\chi U(t)\|_{H^s(\T^{d-1}_1)}\leq C\|U(t)\|_{(s)}$ and $\|\chi h(t)\|_{H^s(\T^{d-1}_1)}\leq C\|h(t)\|_{H^s(\T^{d-1}_1)}$, so that $\int_{I_j}(\|v(t)\|^2_{H^s(\T^{d-1}_1)}+\|h(t)\|^2_{H^s(\T^{d-1}_1)})dt\leq C\lambda^{-1}$.

\emph{Residual.}
We show that $\|\partial_tv\|_{L^2(I_j\times\T^{d-1}_1)}+\|\Dc v\|_{L^2(I_j\times\T^{d-1}_1)}\leq C\mu^{-1/2}$.
For $r=1$, $\partial_tv=\Dc v=\chi\Dc u$, since $I_j\times\mathrm{supp}\,\chi\subset Q_0^\ast$.
For $r=2$, we use the energy inequality on cones \citep[Section~2.4.3, proof of Theorem~6]{evans1998partial}: if $\rho_0>0$, $K=\{(t,y'):|t-t_0|<\rho_0/2,\ |y'-y_0|<\rho_0-|t-t_0|\}$ and $w\in H^2(K)$, then $e(t)=\frac12\int_{K_t}(|\partial_tw|^2+|\nabla'w|^2)dy'$, where $K_t$ is the section of $K$ at time $t$, satisfies $e(t)\leq e(t_0)+|\int_{t_0}^t\|\partial_tw\|_{L^2(K_{t'})}\|(\partial_t^2-\Delta')w\|_{L^2(K_{t'})}dt'|$.
For smooth real $w$, the computation of $e'$ in that proof gives $e'(t)\leq\int_{K_t}\partial_tw\,(\partial_t^2w-\Delta'w)\,dy'$ for $t>t_0$, because the boundary term is nonpositive, and the Cauchy--Schwarz inequality gives the claim; it extends to $w\in H^2(K)$ by density of $C^\infty(\overline K)$ ($K$ is convex), since the traces of $\partial_tw$ and $\nabla'w$ on the sections $K_t$ are bounded in $L^2$ uniformly in $t$; to complex $w$ by adding the inequalities for the real and imaginary parts (Cauchy--Schwarz in $\R^2$); and to $t<t_0$ by time reversal.
Apply it with $t_0=t_j$, $\rho_0=1/8$ and $y_0\in[-17/16,17/16]^{d-1}$: then $K\subset Q_0^\ast$, $(\partial_t^2-\Delta')w=\Dc u$ on $K$, and $e(t_j)=0$ because $h(t_j)=U(t_j)$ and $\partial_th(t_j)=\partial_tU(t_j)$ on $K_{t_j}$.
Hence $\sup_te\leq(2\sup_te)^{1/2}|I_j|^{1/2}\|\Dc u\|_{L^2(Q_0^\ast)}$, so $\sup_te\leq C\mu^{-1}$.
Every section contains $B(y_0,1/16)$, so finitely many centers give $\|\partial_tw\|+\|\nabla'w\|\leq C\mu^{-1/2}$ in $L^2(I_j\times(-17/16,17/16)^{d-1})$, and $w(t)=\int_{t_j}^t\partial_tw$ gives the same bound for $w$.
Since $\partial_tv=\chi\partial_tw$ and $\Dc v=\chi\Dc u-2\nabla'\chi\cdot\nabla'w-w\Delta'\chi$, the claim follows.

\emph{Approximation.}
Write $v=\sum_ka_k(t)\psi_k$ and $h+v=\sum_kc_k(t)\psi_k$ on $I_j\times\T^{d-1}_1$, and let $g_k=a_k'$ if $r=1$ and $g_k=a_k''+\omega_k^2a_k$ if $r=2$, which are the coefficients of $\Dc v$.
Parseval's identity gives $\sum_k(\|a_k'\|^2_{L^2(I_j)}+\|g_k\|^2_{L^2(I_j)})\leq C\mu^{-1}$ and $\sum_k(1+\omega_k^2)^s\|c_k\|^2_{L^2(I_j)}\leq C\lambda^{-1}$.
For an interval $I$, let $T^I_J$ be the $L^2(I)$-projection onto the span of $b_p^I(t)=\cos(\pi p(t-\inf I)/|I|)$, $0\leq p<J$; integrating by parts against the sine basis and Bessel's inequality give $\|a-T_J^Ia\|_{L^2(I)}\leq|I|\|a'\|_{L^2(I)}/(\pi J)$ for $a\in H^1(I)$.
Let $\varepsilon_1\in(0,1]$, $K_0\geq1$, $K=K_0\lambda^{-1/(2s)}$, $J=\lceil\mu^{-1/2}/\varepsilon_1\rceil$ and $J_k=\lceil\mu^{-1/2}/(\varepsilon_1\omega_k)\rceil$.
Let $W_j$ be spanned by the functions $b_p^{I_j}(t)\psi_k(y')$ with $\omega_k\leq K$ and $p<J$ if $r=1$, and, if $r=2$, by $t$, by $b_p^{I_j}(t)$ with $p<J$, and by $e^{\pm i\omega_kt}b_p^{I_j}(t)\psi_k(y')$ with $0<\omega_k\leq K$ and $p<J_k$.
It contains the modes of $h$ with $\omega_k\leq K$, and the modes with $\omega_k>K$ have squared norm at most $C\sum_{\omega_k>K}\|c_k\|^2\leq CK^{-2s}\lambda^{-1}=CK_0^{-2s}$.
If $r=1$ or $k=0$, then $\|a_k-T^{I_j}_Ja_k\|\leq\varepsilon_1\mu^{1/2}\|a_k'\|$.
If $r=2$ and $k\neq0$, then $a_k=e^{i\omega_kt}a_k^++e^{-i\omega_kt}a_k^-$ with $a_k^\pm=\frac12e^{\mp i\omega_kt}(a_k\pm a_k'/(i\omega_k))$ and $(a_k^\pm)'=\pm e^{\mp i\omega_kt}g_k/(2i\omega_k)$, and projecting $a_k^\pm$ with $T^{I_j}_{J_k}$ gives an error at most $|I_j|\|g_k\|/(\pi J_k\omega_k)\leq\varepsilon_1\mu^{1/2}\|g_k\|$.
Summing over $k$, $\mathrm{dist}_{L^2(I_j\times(-1,1)^{d-1})}(q-\Pi_{\bar s}q,W_j)^2\leq C(\varepsilon_1^2+K_0^{-2s})$.

\emph{Conclusion.}
Let $V=\Pol_{\bar s}+\sum_j\{\mathbf 1_{I_j}v:v\in W_j\}$, which does not depend on $q$ (the $W_j$ involve the $I_j$, not the $t_j$).
Summing over the $32$ slabs gives $\mathrm{dist}_{L^2(Q_0)}(q,V)\leq\varepsilon$ for $K_0$ large and $\varepsilon_1$ small.
As at most $CK^{d-1}$ frequencies satisfy $\omega_k\leq K$, $\dim V\leq C+CK^{d-1}(1+\mu^{-1/2})$ if $r=1$, and $\dim V\leq C+CK^{d-1}+C\mu^{-1/2}(1+\sum_{0<\omega_k\leq K}\omega_k^{-1})$ if $r=2$, where the sum is at most $C\log(2+K)$ if $d=2$ and $CK^{d-2}$ if $d\geq3$.
Since $\lambda\leq1$, both bounds are at most $C_\varepsilon(\lambda^{-(d-1)/(2s)}+\lambda^{-(d-r)/(2s)}\mu^{-1/2}\ell_{r,d}(\lambda))$, which proves the lemma.
\end{proof}

\subsection{Local approximation for elliptic operators}
For $\eta>0$, let $\D_\eta=\sum_{|\alpha|\leq r}\eta^{r-|\alpha|}a_\alpha\partial^\alpha$, so that $q(z)=f(x_0+\eta z)$ satisfies $(\D f)(x_0+\eta z)=\eta^{-r}(\D_\eta q)(z)$.
If $\D$ is elliptic, homogeneity and compactness of the unit sphere give $c_{\mathrm e}>0$ with $|P_r(i\zeta)|\geq c_{\mathrm e}|\zeta|^r$ for $\zeta\in\R^d$.

\begin{lem}[Interior estimates]
\label[lem]{lem:interior}
Let $\D$ be elliptic of order $r$ with constant real coefficients, $\eta_{\max}>0$ and $j_0\geq0$ an integer.
Uniformly in $\eta\in(0,\eta_{\max}]$, if $g\in L^2(Q_0^\ast)$ and $\D_\eta g\in H^{j_0}(Q_0^\ast)$, then $g\in H^{j_0+r}(Q_0)$ and $\|g\|_{H^{j_0+r}(Q_0)}\leq C_{j_0,\eta_{\max}}(\|\D_\eta g\|_{H^{j_0}(Q_0^\ast)}+\|g\|_{L^2(Q_0^\ast)})$.
\end{lem}
\begin{proof}
Let $\D_0=\sum_{|\alpha|=r}a_\alpha\partial^\alpha$.
Since $(1+|\zeta|^2)^{r/2}\leq C(1+|P_r(i\zeta)|)$, every compactly supported $v\in L^2(\R^d)$ satisfies, for every $t\in\R$, $\|v\|_{(t+r)}\leq C(\|\D_0v\|_{(t)}+\|v\|_{(t)})$ whenever the right-hand side is finite.
Let $J=j_0+r$ and $U_i=(\frac65-\frac i{5J})Q_0$, so that $U_0=Q_0^\ast$ and $U_J=Q_0$.
Assume $g\in H^i(U_i)$ for some $i<J$, take $\chi\in C_c^\infty(U_i)$ equal to $1$ on $U_{i+1}$, and let $v=\chi g$, extended by zero, so that $\|v\|_{(i)}\leq C\|g\|_{H^i(U_i)}$.
By the Leibniz rule, $\D_0v=\chi\D_\eta g+(\D_0-\D_\eta)v+\sum_{|\beta|\leq r-1}\partial^\beta(\varphi_\beta g)$, where the $\varphi_\beta\in C_c^\infty(U_i)$ have derivatives bounded uniformly in $\eta\leq\eta_{\max}$.
For $t=i-r+1\leq j_0$, we have $\|\chi\D_\eta g\|_{(t)}\leq C\|\D_\eta g\|_{H^{j_0}(Q_0^\ast)}$, $\|\partial^\beta(\varphi_\beta g)\|_{(t)}\leq\|\varphi_\beta g\|_{(i)}\leq C\|g\|_{H^i(U_i)}$ because $t+|\beta|\leq i$, and $\|(\D_0-\D_\eta)v\|_{(t)}\leq C\|v\|_{(i)}$ because $\D_0-\D_\eta$ has order at most $r-1$ and bounded coefficients.
Hence $\|g\|_{H^{i+1}(U_{i+1})}\leq C\|v\|_{(i+1)}\leq C(\|\D_\eta g\|_{H^{j_0}(Q_0^\ast)}+\|g\|_{H^i(U_i)})$, and iterating from $i=0$ to $J-1$ proves the lemma.
\end{proof}

\begin{cor}[Periodic extension]
\label[cor]{cor:extension}
Let $L>0$ and $U$ be a bounded Lipschitz domain with $\overline U\subset(-2L,2L)^q$, for instance $U=\Omega$ with the $L$ of (A1).
For every real $t\geq0$, every $f\in H^t(U)$ coincides on $U$ with a $4L$-periodic function $g$ such that $\|g\|_{H^t(\T^q_L)}\leq C_{U,t}\|f\|_{H^t(U)}$.
\end{cor}
\begin{proof}
Let $\chi\in C_c^\infty((-2L,2L)^q)$ equal $1$ near $\overline U$.
The periodization $g$ of $\chi Ef$ agrees with $f$ on $U$, and \cref{lem:fourier}(i) and the extension bound give $\|g\|_{H^t(\T^q_L)}\leq C\|Ef\|_{H^t(\R^q)}\leq C_{U,t}\|f\|_{H^t(U)}$.
\end{proof}

\begin{lem}[Local approximation, elliptic case]
\label[lem]{lem:localell}
Let $\D$ be elliptic of order $r$ with constant real coefficients, $s\geq r$, $\eta_{\max}>0$, $j_0\geq1$ an integer and $\varepsilon>0$.
There is $C$, depending also on $\varepsilon$, $j_0$ and $\eta_{\max}$, such that, for every $\eta\in(0,\eta_{\max}]$, $\lambda\in(0,1]$ and $\mu>0$, there is a linear space $V$ of functions on $Q_0$ with $\dim V\leq C(1+\lambda^{-d/(2(j_0+r))}+\mu^{-d/(2r)})$ such that $\mathrm{dist}_{L^2(Q_0)}(q,V)^2\leq\varepsilon^2(\lambda|q|^2_{s,Q_0^\ast}+\mu\|\D_\eta q\|^2_{L^2(Q_0^\ast)})$ for every $q\in H^s(Q_0^\ast)$.
\end{lem}
\begin{proof}
Constants may depend on $\varepsilon$, $j_0$ and $\eta_{\max}$.
As in \cref{lem:local}, $V$ will contain $\Pol_{\bar s}$ with $\bar s=\lceil s\rceil$, so we may assume $\lambda|q|^2_{s,Q_0^\ast}+\mu\|\D_\eta q\|^2_{L^2(Q_0^\ast)}=1$.
Let $u=q-\Pi_{\bar s}q$, $g=\D_\eta q$ and $R=\D_\eta\Pi_{\bar s}q=g-\D_\eta u\in\Pol_{\bar s}$.
By \cref{lem:poly}(a), $\|u\|_{H^s(Q_0^\ast)}\leq C\lambda^{-1/2}$, so $\|R\|_{L^2(Q_0^\ast)}\leq\|g\|_{L^2(Q_0^\ast)}+C\|u\|_{H^r(Q_0^\ast)}\leq C(\lambda^{-1/2}+\mu^{-1/2})$.
Let $\tilde g\in L^2(\T^d_1)$ be the periodization of $g$ extended by zero.
The symbol $P_\eta(z)=\sum_{|\alpha|\leq r}\eta^{r-|\alpha|}a_\alpha z^\alpha$ satisfies $|P_\eta(i\zeta)-P_r(i\zeta)|\leq C(1+|\zeta|^{r-1})$, so there is $k_0$, independent of $\eta$, with $|P_\eta(i\pi k/2)|\geq c_{\mathrm e}|\pi k/2|^r/2$ for $|k|>k_0$.
Define $Z$ on $\T^d_1$ by $\hat Z_k=\hat{\tilde g}_k/P_\eta(i\pi k/2)$ for $|k|>k_0$ and $\hat Z_k=0$ otherwise.
Then $\|Z\|_{H^r(\T^d_1)}\leq C\|g\|_{L^2(Q_0^\ast)}\leq C\mu^{-1/2}$, and $\D_\eta Z=\tilde g-\tilde g_0$, where $\tilde g_0$ keeps the frequencies $|k|\leq k_0$ of $\tilde g$.
On $Q_0^\ast$, $h=u-Z$ satisfies $\D_\eta h=\tilde g_0-R$, a trigonometric polynomial with at most $C$ coefficients bounded by $C\|g\|_{L^2(Q_0^\ast)}$ plus an element of $\Pol_{\bar s}$, so $\|\D_\eta h\|_{H^{j_0}(Q_0^\ast)}+\|h\|_{L^2(Q_0^\ast)}\leq C(\lambda^{-1/2}+\mu^{-1/2})$.
\Cref{lem:interior,cor:extension} give $\bar h\in H^{j_0+r}(\T^d_1)$, equal to $h$ on $Q_0$, with $\|\bar h\|_{H^{j_0+r}(\T^d_1)}\leq C(\lambda^{-1/2}+\mu^{-1/2})$.
Let $T_J$ keep the frequencies $\|k\|_\infty\leq J$ on $\T^d_1$; the truncation bound of \cref{app:tools} gives $\|Z-T_JZ\|_{L^2(\T^d_1)}\leq CJ^{-r}\mu^{-1/2}$ and $\|\bar h-T_J\bar h\|_{L^2(\T^d_1)}\leq CJ^{-j_0-r}(\lambda^{-1/2}+\mu^{-1/2})$.
Both are at most $\varepsilon/2$ for some $J\leq C(1+\mu^{-1/(2r)}+(\lambda^{-1/2}+\mu^{-1/2})^{1/(j_0+r)})$.
Since $q=\Pi_{\bar s}q+Z+\bar h$ on $Q_0$, the space $V=\Pol_{\bar s}+\mathrm{span}\{e^{i\pi\langle k,\cdot\rangle/2}|_{Q_0}:\|k\|_\infty\leq J\}$, which does not depend on $q$, approximates $q$ within $\varepsilon$.
Its dimension is at most $C+(2J+1)^d\leq C(1+\lambda^{-d/(2(j_0+r))}+\mu^{-d/(2r)}+\mu^{-d/(2(j_0+r))})$, and $\mu^{-d/(2(j_0+r))}\leq1+\mu^{-d/(2r)}$, which proves the lemma.
\end{proof}

\subsection{Whitney assembly}
\begin{lem}[Whitney assembly]
\label[lem]{lem:boundary}
Let $U\subseteq\R^d$ be a bounded Lipschitz domain, $U'=\{y:\mathrm{dist}(y,U)<\sqrt d\}$, $s\geq1$, $\lambda\in(0,1]$, $\mu\geq0$, $\mathcal L$ a linear differential operator of order at most $s$ with constant coefficients, and $\varepsilon>0$.
There are $\bar c\geq2$ and $\varepsilon'>0$, depending only on $\varepsilon$, $d$, $s$ and $U$, with the following property.
Assume that, for every dyadic $h\geq\bar c\lambda^{1/(2s)}$ and every cube $Q$ of side $h$ with $Q^\ast\subseteq U$, there is a linear space $V_Q$ of functions on $Q$ with $\dim V_Q\leq N_h$ such that $\mathrm{dist}_{L^2(Q)}(F,V_Q)^2\leq\varepsilon'^2(\lambda|F|^2_{s,Q^\ast}+\mu\|\mathcal LF\|^2_{L^2(Q^\ast)})$ for every $F\in H^s(Q^\ast)$.
Then there is a linear space $W\subseteq L^2(U)$ with $\dim W\leq C_\varepsilon(\lambda^{-(d-1)/(2s)}+\sum_hh^{1-d}N_h)$, the sum being over dyadic $h\in[\bar c\lambda^{1/(2s)},\mathrm{diam}(U)]$, such that $\mathrm{dist}_{L^2(U)}(F,W)^2\leq\varepsilon^2(\lambda|F|^2_{s,U'}+\mu\|\mathcal LF\|^2_{L^2(U)})$ for every $F\in H^s(U')$.
\end{lem}
\begin{proof}
By the Whitney decomposition \citep[Chapter~VI, \S1, Theorem~1 and Proposition~3]{stein1970lipschitz}, $U$ is the union of closed dyadic cubes $Q$ with disjoint interiors and $\mathrm{diam}(Q)\leq\mathrm{dist}(Q,U^c)\leq4\,\mathrm{diam}(Q)$, the cubes $Q^\ast$ lie in $U$, and each point lies in at most $12^d$ of them.
A cube of side $h$ thus lies within $5\sqrt d\,h$ of $\partial U$, and $|\{x:\mathrm{dist}(x,\partial U)\leq t\}|\leq C_{U,T}t$ for $t\leq T$ (cover $\partial U$ by finitely many Lipschitz graphs); hence there are at most $C_Uh^{1-d}$ Whitney cubes of side $h$.
\emph{Interior.}
Let $W_{\mathrm i}$ be the sum of the spaces $V_Q$, extended by zero, over the Whitney cubes $Q$ of side $h_Q\geq\bar c\lambda^{1/(2s)}$.
Since each pair of points lies in at most $12^d$ cubes $Q^\ast$, also in the double integrals of noninteger order, $\sum_Q|F|^2_{s,Q^\ast}\leq12^d|F|^2_{s,U}$ and $\sum_Q\|\mathcal LF\|^2_{L^2(Q^\ast)}\leq12^d\|\mathcal LF\|^2_{L^2(U)}$, so the squared distance from $F$ to $W_{\mathrm i}$ on the union of these cubes is at most $12^d\varepsilon'^2(\lambda|F|^2_{s,U}+\mu\|\mathcal LF\|^2_{L^2(U)})$, and $\dim W_{\mathrm i}\leq C_U\sum_hh^{1-d}N_h$.
\emph{Boundary layer.}
The union $U_\partial$ of the other Whitney cubes lies within $5\sqrt d\,\bar c\lambda^{1/(2s)}$ of $\partial U$.
The cubes $P$ of a grid of side $\lambda^{1/(2s)}/\bar c\leq1/2$ that meet $U_\partial$ lie in $U'$ and within $6\sqrt d\,\bar c\lambda^{1/(2s)}$ of $\partial U$, so there are at most $C_{U,\bar c}\lambda^{-(d-1)/(2s)}$ of them.
Let $W_\partial$ be the sum of the spaces $\{\mathbf 1_{P\cap U_\partial}p:p\in\Pol_{\lceil s\rceil}\}$.
By \cref{lem:poly}(a) and since the $P$ have disjoint interiors, $\mathrm{dist}_{L^2(U_\partial)}(F,W_\partial)^2\leq C\bar c^{-2s}\lambda\sum_P|F|^2_{s,P}\leq C\bar c^{-2s}\lambda|F|^2_{s,U'}$.
\emph{Conclusion.}
The space $W=W_{\mathrm i}+W_\partial$ satisfies $\mathrm{dist}_{L^2(U)}(F,W)^2\leq(12^d\varepsilon'^2+C\bar c^{-2s})(\lambda|F|^2_{s,U'}+\mu\|\mathcal LF\|^2_{L^2(U)})$; choosing $\bar c$ large and then $\varepsilon'$ small proves the lemma.
\end{proof}

\subsection{Approximation for canonical hyperbolic operators}
\begin{prop}[Approximation on a domain, canonical hyperbolic operators]
\label[prop]{prop:canonical}
Let $r\in\{1,2\}$, $s\geq r$, $U\subseteq\R^d$ a bounded Lipschitz domain with $U'$ as in \cref{lem:boundary}, and $\varepsilon>0$.
For $\lambda\in(0,1]$ and $\mu>0$, there is a linear space $W\subseteq L^2(U)$ such that $\mathrm{dist}_{L^2(U)}(F,W)^2\leq\varepsilon^2(\lambda|F|^2_{s,U'}+\mu\|\Dc F\|^2_{L^2(U)})$ for every $F\in H^s(U')$, and $\dim W\leq C_{U,\varepsilon}(\lambda^{-(d-1)/(2s)}\log(2+1/\lambda)+\lambda^{-(d-r)/(2s)}\mu^{-1/2}\ell_{r,d}(\lambda))$.
\end{prop}
\begin{proof}
Let $\bar c,\varepsilon'$ be given by \cref{lem:boundary} with $\mathcal L=\Dc$, and let $Q$ be a cube of side $h\geq\bar c\lambda^{1/(2s)}$ and center $y_Q$ with $Q^\ast\subseteq U$.
Let $\lambda_Q=(h/2)^{-2s}\lambda\leq1$ and $\mu_Q=(h/2)^{-2r}\mu$; for $q(z)=F(y_Q+(h/2)z)$, squared $L^2$ norms on $Q_0$, $\lambda_Q|q|^2_{s,Q_0^\ast}$ and $\mu_Q\|\Dc q\|^2_{L^2(Q_0^\ast)}$ equal $(h/2)^{-d}$ times squared $L^2$ norms on $Q$, $\lambda|F|^2_{s,Q^\ast}$ and $\mu\|\Dc F\|^2_{L^2(Q^\ast)}$.
Hence \cref{lem:local}, with $\varepsilon'$ and the parameters $(\lambda_Q,\mu_Q)$, gives a space $V_Q$ as required in \cref{lem:boundary}, with $\dim V_Q\leq N_h=C(h^{d-1}\lambda^{-(d-1)/(2s)}+h^d\lambda^{-(d-r)/(2s)}\mu^{-1/2}\ell_{r,d}(\lambda))$, where $\lambda_Q\geq(\mathrm{diam}(U)/2)^{-2s}\lambda$ bounds the logarithm.
There are at most $C\log(2+1/\lambda)$ dyadic scales $h\in[\bar c\lambda^{1/(2s)},\mathrm{diam}(U)]$, and their sum is at most $2\,\mathrm{diam}(U)$, so $\sum_hh^{1-d}N_h$ is bounded by the claimed dimension, and \cref{lem:boundary} proves the proposition.
\end{proof}

\subsection{Proof of \texorpdfstring{\cref{thm:count}}{the effective dimension theorem}}
\begin{proof}
\emph{(i) Sobolev count.}
Since $\theta^\ast G\theta=\|f_\theta\|^2_{L^2(\Omega)}\leq\|f_\theta\|^2_{L^2(\T)}=(4L)^d\|\theta\|_2^2$, (A1) gives $\Sigma\preceq\kappa G\preceq\kappa(4L)^dI$, and \cref{lem:counting}(i) yields $\mathcal N(\lambda,0)\leq\sum_{k\in\Z^d}\kappa(4L)^d/(\kappa(4L)^d+\lambda(1+\|k\|_2^{2s}))\leq\sum_{k\in\Z^d}\min(1,C\lambda^{-1}\|k\|_2^{-2s})\leq C\lambda^{-d/(2s)}$, because $2s>d$.

\emph{(ii) Hyperbolic operators.}
Let $\lambda,\mu\in(0,1]$, $\tau\in(0,1]$, and let $\theta$ satisfy $\tau\lambda\|S\theta\|_2^2+\tau\mu\|\D(f_\theta)\|^2_{L^2(\Omega)}\leq1$ and $\kappa\|f_\theta\|^2_{L^2(\Omega)}\leq1$; write $f=f_\theta$, so that $\tau\lambda\|f\|^2_{H^s(\T)}\leq C$.
Let $\beta,\mathbf M,c_0$ be given by \cref{lem:normal}, $U=\mathbf M^{-1}\Omega$, which is a bounded Lipschitz domain because linear maps preserve the uniform cone property, which characterizes bounded Lipschitz domains \citep[Theorem~1.2.2.2]{grisvard1985elliptic}, $U'$ as in \cref{lem:boundary}, and $F(y)=e^{-\langle\beta,\mathbf My\rangle}f(\mathbf My)$.
Choose $w,w_2\in C_c^\infty(\R^d)$ with $w=1$ on $\mathbf MU'$ and $w_2=1$ on $\mathrm{supp}\,w$; then $F=((e^{-\langle\beta,\cdot\rangle}w)(w_2f))\circ\mathbf M$ on $U'$, and \cref{lem:fourier}(ii)--(iii) give $\|F\|_{H^s(U')}\leq C\|f\|_{H^s(\T)}$.
By \cref{lem:normal}, $\Dc F(y)=e^{-\langle\beta,\mathbf My\rangle}(\D(f)-c_0f)(\mathbf My)$, so, as $\tau\mu\leq1$, $\tau\mu\|\Dc F\|^2_{L^2(U)}\leq C(\tau\mu\|\D(f)\|^2_{L^2(\Omega)}+\|f\|^2_{L^2(\Omega)})\leq C(1+1/\kappa)$.
Thus $\tau\lambda|F|^2_{s,U'}+\tau\mu\|\Dc F\|^2_{L^2(U)}\leq C$, and \cref{prop:canonical} with parameters $(\tau\lambda,\tau\mu)$ gives $W_U$, independent of $\theta$, with $\mathrm{dist}_{L^2(U)}(F,W_U)^2\leq C\varepsilon^2$; then $W_\tau=\{x\mapsto e^{\langle\beta,x\rangle}v(\mathbf M^{-1}x):v\in W_U\}$ has the same dimension, and changing variables gives $\mathrm{dist}_{L^2(\Omega)}(f,W_\tau)^2\leq C\varepsilon^2\leq1/(3\kappa)$ for $\varepsilon$ small enough.
Since $\log(2+1/(\tau\lambda))\leq C\log(2+1/\lambda)(1+\log(1/\tau))$, the dimension of $W_\tau$ is at most $C(1+\log(1/\tau))(\tau^{-a_1}\lambda^{-(d-1)/(2s)}\log(2+1/\lambda)+\tau^{-a_2}\lambda^{-(d-r)/(2s)}\mu^{-1/2}\ell_{r,d}(\lambda))$, with $a_1=(d-1)/(2s)$ and $a_2=1/2+(d-r)/(2s)$.
Both exponents are less than one, since $s>\max(r,d-r)\geq d/2$ and $s>d-r$, so \cref{lem:counting}(ii) proves (ii).

\emph{(iii) Elliptic operators.}
Let $\mu>0$, $\tau\in(0,1]$, $\lambda'=\tau\lambda$ and $\mu'=\tau\mu$, and let $f=f_\theta$ with $\lambda'\|S\theta\|_2^2+\mu'\|\D(f)\|^2_{L^2(\Omega)}\leq1$.
With $\Omega'=\{x:\mathrm{dist}(x,\Omega)<\sqrt d\}$, \cref{lem:fourier}(ii) gives $\lambda'|f|^2_{s,\Omega'}\leq C\lambda'\|f\|^2_{H^s(\T)}\leq C$.
Choose an integer $j_0\geq1$ with $sd/(j_0+r)\leq(d-1)/2$, let $\varepsilon>0$, and let $\bar c,\varepsilon'$ be given by \cref{lem:boundary} with $U=\Omega$ and $\mathcal L=\D$.
For a cube $Q$ of side $h\geq\bar c\lambda'^{1/(2s)}$ and center $x_Q$ with $Q^\ast\subseteq\Omega$, the function $q(z)=f(x_Q+(h/2)z)$ satisfies $(\D f)(x_Q+(h/2)z)=(h/2)^{-r}(\D_{h/2}q)(z)$, with $h/2\leq\mathrm{diam}(\Omega)/2$.
As in the proof of \cref{prop:canonical}, \cref{lem:localell} with $\varepsilon'$, $\eta=h/2$ and the parameters $(h/2)^{-2s}\lambda'\leq1$ and $(h/2)^{-2r}\mu'$ gives a space $V_Q$ as required in \cref{lem:boundary}, with $\dim V_Q\leq N_h=C(1+x^{sd/(j_0+r)}+h^d\mu'^{-d/(2r)})$, where $x=(h/2)\lambda'^{-1/(2s)}\geq1$, so that $x^{sd/(j_0+r)}\leq x^{(d-1)/2}$.
Since $d\geq2$, the sums $\sum_hh^{1-d}$ and $\sum_hh^{1-d}x^{(d-1)/2}$ over dyadic $h\geq\bar c\lambda'^{1/(2s)}$ are geometric and at most $C\lambda'^{-(d-1)/(2s)}$, and $\sum_hh\leq2\,\mathrm{diam}(\Omega)$.
\Cref{lem:boundary} thus gives $W_\tau$, independent of $\theta$, with $\mathrm{dist}_{L^2(\Omega)}(f,W_\tau)^2\leq C\varepsilon^2\leq1/(3\kappa)$ for $\varepsilon$ small enough, and $\dim W_\tau\leq C(\tau^{-(d-1)/(2s)}\lambda^{-(d-1)/(2s)}+\tau^{-d/(2r)}\mu^{-d/(2r)})$.
Both exponents of $\tau$ are less than one, since $d<2r\leq2s$, so \cref{lem:counting}(ii) proves (iii).

All constants are independent of $\lambda$, $\mu$ and $m$, which proves (i)--(iii); finally, $\mathcal N(\lambda,\mu)=\mathrm{tr}(\Sigma^{1/2}(\Sigma+\lambda S^2+\mu D^\ast GD)^{-1}\Sigma^{1/2})\leq\mathcal N(\lambda,0)$, because the inverse is nonincreasing for the Loewner order.
\end{proof}

\section{Upper bounds}
\label{app:upper}

\subsection{Upper bound for given parameters}
\begin{lem}[Oracle inequality]
\label[lem]{lem:oracle}
Assume (A1)--(A2), let $\D$ be linear with constant real coefficients, $s>d/2$, $\lambda>0$ and $\mu\geq0$.
Let $g\in L^2(\Omega)$ be known and real-valued, and replace $\|\D(f_\theta)\|^2_{L^2(\Omega)}$ by $\|\D(f_\theta)-g\|^2_{L^2(\Omega)}$ in~\eqref{eq:estimator} ($g=0$ gives the original estimator).
For every fixed $\theta$,
\begin{multline*}
\EE\|\hat f_n-f^\star\|^2_{L^2(\PP_X)}
\leq\Big(1+\frac{a_s}{\lambda n}\Big)^2
\Big(\|f^\star-f_\theta\|^2_{L^2(\PP_X)}+\lambda\|S\theta\|_2^2+\mu\|\D(f_\theta)-g\|^2_{L^2(\Omega)}\Big)\\
+\Big(1+\frac{a_s}{\lambda n}\Big)\frac{\sigma^2}n\mathcal N(\lambda,\mu).
\end{multline*}
\end{lem}
\begin{proof}
The argument is the leave-one-out analysis of \citet{mourtada2022an}, adapted to a general quadratic penalty and to a misspecified target.
Let $B=\lambda S^2+\mu D^\ast GD$, and let $\hat\theta_0$ be the minimizer with responses $f^\star(X_i)$ in place of $Y_i$.
Subtracting the normal equations gives $\hat\theta=\hat\theta_0+\hat\theta_\varepsilon$ with $\hat\theta_\varepsilon=(\hat\Sigma+B)^{-1}n^{-1}\sum_i\varepsilon_i\phi(X_i)$.
Since the pairs $(X_i,\varepsilon_i)$ are independent, $\EE(\varepsilon_i\mid X_1,\ldots,X_n)=0$ by (A2), so $\hat\theta_\varepsilon$ has conditional mean zero and $\EE\|\hat f_n-f^\star\|^2_{L^2(\PP_X)}=\EE\|f_{\hat\theta_0}-f^\star\|^2_{L^2(\PP_X)}+\EE\|f_{\hat\theta_\varepsilon}\|^2_{L^2(\PP_X)}$.

Let $X_{n+1}$ be an independent input, $\phi_i=\phi(X_i)$, $A_+=\sum_{j\leq n+1}\phi_j\phi_j^\ast+nB$, $A_{-i}=A_+-\phi_i\phi_i^\ast$ and $h_i=\phi_i^\ast A_{-i}^{-1}\phi_i$.
Since $A_{-i}\succeq n\lambda S^2$, $0\leq h_i\leq a_s/(\lambda n)$, and the Sherman--Morrison formula gives $h_i=(1+h_i)\phi_i^\ast A_+^{-1}\phi_i$.

\emph{Noise.}
By (A2) and $\hat\Sigma\preceq\hat\Sigma+B$, $\EE\|f_{\hat\theta_\varepsilon}\|^2_{L^2(\PP_X)}\leq\frac{\sigma^2}n\EE\,\mathrm{tr}((\hat\Sigma+B)^{-1}\Sigma)=\sigma^2\EE h_{n+1}$.
Exchangeability and the Sherman--Morrison identity give $\EE h_{n+1}\leq\frac{1+a_s/(\lambda n)}{n+1}\EE\,\mathrm{tr}(A_+^{-1}(A_+-nB))$.
The map $A\mapsto\mathrm{tr}(A^{-1}B)$ is convex on positive definite matrices, since $u^\ast A^{-1}u=\sup_v(2\mathrm{Re}(u^\ast v)-v^\ast Av)$ and $B=\sum_ju_ju_j^\ast$; hence $A\mapsto\mathrm{tr}(A^{-1}(A-nB))$ is concave, and Jensen's inequality with $\EE A_+=(n+1)\Sigma+nB\succeq n(\Sigma+B)$ gives $\EE\,\mathrm{tr}(A_+^{-1}(A_+-nB))\leq(n+1)\mathrm{tr}(((n+1)\Sigma+nB)^{-1}\Sigma)\leq\frac{n+1}n\mathcal N(\lambda,\mu)$, so that $\EE\|f_{\hat\theta_\varepsilon}\|^2_{L^2(\PP_X)}\leq(1+\frac{a_s}{\lambda n})\frac{\sigma^2}n\mathcal N(\lambda,\mu)$.

\emph{Noiseless fit.}
Let $y_i=f^\star(X_i)$, let $\theta_+$ minimize $\sum_{j\leq n+1}|y_j-f_\theta(X_j)|^2+n\lambda\|S\theta\|_2^2+n\mu\|\D(f_\theta)-g\|^2_{L^2(\Omega)}$, and let $\theta_{-i}$ minimize the same objective without observation $i$.
Then $(\theta_{-i},X_i)$ has the law of $(\hat\theta_0,X_{n+1})$.
Subtracting the normal equations, in which the source terms cancel, gives $A_+(\theta_+-\theta_{-i})=\phi_i(y_i-\phi_i^\ast\theta_{-i})$, hence $y_i-\phi_i^\ast\theta_{-i}=(1+h_i)(y_i-\phi_i^\ast\theta_+)$.
Averaging over $i$ and comparing the objective at $\theta_+$ with its value at $\theta$ gives
\begin{align*}
\EE\|f_{\hat\theta_0}-f^\star\|^2_{L^2(\PP_X)}
&\leq\Big(1+\frac{a_s}{\lambda n}\Big)^2\EE\frac1{n+1}\sum_{i\leq n+1}|y_i-\phi_i^\ast\theta_+|^2\\
&\leq\Big(1+\frac{a_s}{\lambda n}\Big)^2\Big(\|f^\star-f_\theta\|^2_{L^2(\PP_X)}+\lambda\|S\theta\|_2^2+\mu\|\D(f_\theta)-g\|^2_{L^2(\Omega)}\Big).
\end{align*}
Adding the noise bound proves the oracle inequality for every fixed $\theta$.
\end{proof}

\begin{lem}[Upper bound for given parameters]
\label[lem]{lem:upper}
Assume (A1)--(A2), let $\D$ be linear of order $r$ with constant real coefficients, let $s>\max(r,d/2)$, and let $g\in L^2(\Omega)$ be known and real-valued.
There is $C>0$ such that, for all $n\geq1$, $\rho,\sigma>0$, $\delta\geq0$, $\lambda\in[1/n,1]$, $\mu\geq0$ and $m\geq\max\{\lambda^{-1/(2s)},(\mu/\lambda)^{1/(2s-2r)}\}$, the estimator $\hat f_n$ of \cref{lem:oracle} satisfies, with $\F_g(\rho,\delta)=\{f\in\Hc(\rho):\|\D(f)-g\|_{L^2(\Omega)}\leq\delta\}$,
\[
\sup_{f^\star\in\F_g(\rho,\delta)}\EE\|\hat f_n-f^\star\|^2_{L^2(\PP_X)}\leq C\Big(\lambda\rho^2+\mu\delta^2+\frac{\sigma^2}n\mathcal N(\lambda,\mu)\Big).
\]
For $g=0$, $\F_0(\rho,\delta)=\F(\rho,\delta)$ and $\hat f_n$ is the estimator~\eqref{eq:estimator}.
\end{lem}
\begin{proof}
Fix $f^\star\in\F_g(\rho,\delta)$.
By \cref{cor:extension}, $f^\star$ coincides on $\Omega$ with a periodic function $\bar f$ such that $\|\bar f\|_{H^s(\T)}\leq C\rho$; let $\theta=(\hat{\bar f}_k)_{\|k\|_\infty\leq m}$, so that $f_\theta$ is the truncation of $\bar f$.
The norm equivalence and the truncation bound of \cref{app:tools} give $\|S\theta\|_2^2\leq C\rho^2$ and $\|f^\star-f_\theta\|^2_{L^2(\PP_X)}\leq\kappa\|\bar f-f_\theta\|^2_{L^2(\T)}\leq C\rho^2m^{-2s}$.
Since $\bar f$ and $f^\star$ have the same weak derivatives on $\Omega$, the bound on $\D$ of \cref{app:tools} gives $\|\D(f_\theta)-g\|_{L^2(\Omega)}\leq\|\D(f^\star)-g\|_{L^2(\Omega)}+C\|\bar f-f_\theta\|_{H^r(\T)}\leq\delta+C\rho m^{r-s}$.
The choice of $m$ gives $m^{-2s}\leq\lambda$ and $\mu m^{2r-2s}\leq\lambda$, hence $\|f^\star-f_\theta\|^2_{L^2(\PP_X)}+\lambda\|S\theta\|_2^2+\mu\|\D(f_\theta)-g\|^2_{L^2(\Omega)}\leq C(\lambda\rho^2+\mu\delta^2)$.
Since $\lambda n\geq1$, the factor $1+a_s/(\lambda n)$ of \cref{lem:oracle} is at most $1+a_s$.
Thus \cref{lem:oracle} bounds the risk at $f^\star$ by $C(\lambda\rho^2+\mu\delta^2+\sigma^2\mathcal N(\lambda,\mu)/n)$, with $C$ independent of $f^\star$, which proves the lemma.
\end{proof}

\subsection{Proof of \texorpdfstring{\cref{thm:elliptic}}{the elliptic theorem} and of the upper bound of \texorpdfstring{\cref{thm:hyperbolic}}{the hyperbolic theorem}}
\begin{proof}
Both proofs apply \cref{lem:upper} with $g=0$ and \cref{thm:count}, whose hypotheses hold because $s>r>d/2$ in \cref{thm:elliptic} and $s>\max(r,d-r)\geq d/2$ in \cref{thm:hyperbolic}.
Take $C_0\geq1$, so that $\sigma^2/(n\rho^2)\leq1$.
For $q\in\{d-1,d\}$, the definition $\gamma(s,q)=2s/(2s+q)$ and the sample-size assumption give
\begin{equation}
\frac{\sigma^2}n\Big(\frac{\sigma^2}{n\rho^2}\Big)^{-\gamma(s,q)q/(2s)}=\rho^2\Big(\frac{\sigma^2}{n\rho^2}\Big)^{\gamma(s,q)},
\qquad
\frac1n\leq\Big(\frac{\sigma^2}{n\rho^2}\Big)^{\gamma(s,d-1)}\leq\Big(\frac{\sigma^2}{n\rho^2}\Big)^{\gamma(s,d)}\leq1.
\label{eq:id1}
\end{equation}

\emph{Case (b) of both theorems.}
Here $\mu=0$ and $\lambda=(\sigma^2/(n\rho^2))^{\gamma(s,d)}\in[1/n,1]$ by~\eqref{eq:id1}.
By \cref{lem:upper}, \cref{thm:count}(i) and~\eqref{eq:id1} with $q=d$, the risk is at most $C(\lambda\rho^2+\sigma^2\lambda^{-d/(2s)}/n)=2C\rho^2(\sigma^2/(n\rho^2))^{\gamma(s,d)}$, which is the bound of (b).

\emph{Case (a) of both theorems.}
In the elliptic case, let $\delta'=\delta$, $p=2d/(2r+d)$, $e=2r/(2r+d)$ and $\lambda_\delta=\rho^{-2}\delta^{2d/(2r+d)}(\sigma^2/n)^{\gamma(r,d)}=(\delta'/\rho)^p(\sigma^2/(n\rho^2))^e$.
In the hyperbolic case, let $\delta'=\delta\ell$, $p=e=2s/(3s+d-r)$ and $\lambda_\delta=(\sigma^2\delta\ell/(n\rho^3))^p=(\delta'/\rho)^p(\sigma^2/(n\rho^2))^e$.
In both cases $\lambda=\max\{(\sigma^2/(n\rho^2))^{\gamma(s,d-1)},\lambda_\delta\}$, and $\rho^2\lambda_\delta$ is the residual-dependent term of the bound to prove.
Since $\delta_n/\rho=(\sigma^2/(n\rho^2))^{(s-r)/(2s+d)}$ and a direct computation gives $e=(2s-p(s-r))/(2s+d)$ in both cases,
\begin{equation}
\lambda_\delta=\Big(\frac{\sigma^2}{n\rho^2}\Big)^{\gamma(s,d)}\Big(\frac{\delta'}{\delta_n}\Big)^p.
\label{eq:id2}
\end{equation}
As $\delta'\leq\delta_n$, \eqref{eq:id1} and~\eqref{eq:id2} give $1/n\leq\lambda$ and $\lambda\leq(\sigma^2/(n\rho^2))^{\gamma(s,d)}\leq1$.
Moreover, $\lambda\geq(\sigma^2/(n\rho^2))^{\gamma(s,d-1)}\geq\sigma^2/(n\rho^2)$, so $\log(2+1/\lambda)\leq\log(2+n\rho^2/\sigma^2)$ and $\ell_{r,d}(\lambda)\leq\ell$; also $\mu>0$, $\mu\delta^2\leq\lambda\rho^2$, and $\mu\leq\lambda^{(r-1)/s}\leq1$ in the hyperbolic case.
Let $\Lambda=1$ and $V=\mu^{-d/(2r)}$ in the elliptic case, and $\Lambda=\log(2+n\rho^2/\sigma^2)$ and $V=\lambda^{-(d-r)/(2s)}\mu^{-1/2}\ell$ in the hyperbolic case.
By \cref{lem:upper} and \cref{thm:count}(iii), resp.\ (ii), the risk is at most $C(\lambda\rho^2+\Lambda\sigma^2\lambda^{-(d-1)/(2s)}/n+\sigma^2V/n)$.
By~\eqref{eq:id1} with $q=d-1$ and $\lambda\geq(\sigma^2/(n\rho^2))^{\gamma(s,d-1)}$, $\sigma^2\lambda^{-(d-1)/(2s)}/n\leq\rho^2(\sigma^2/(n\rho^2))^{\gamma(s,d-1)}$.
If $\mu$ equals its first cap, $\lambda^{r(d-1)/(sd)}$, resp.\ $\lambda^{(r-1)/s}$, then $V=\lambda^{-(d-1)/(2s)}$, resp.\ $V=\lambda^{-(d-1)/(2s)}\ell$ with $\ell\leq\Lambda$, so $\sigma^2V/n\leq\Lambda\rho^2(\sigma^2/(n\rho^2))^{\gamma(s,d-1)}$.
Otherwise $\delta>0$ and $\mu=\lambda\rho^2/\delta^2$, so $\sigma^2V/n$ equals $(\sigma^2/n)(\delta^2/(\lambda\rho^2))^{d/(2r)}$ in the elliptic case and $(\sigma^2\delta\ell/(n\rho))\lambda^{-(s+d-r)/(2s)}$ in the hyperbolic case.
This is at most $\rho^2\lambda$, since this inequality is equivalent to $\lambda^{(2r+d)/(2r)}\geq\lambda_\delta^{(2r+d)/(2r)}$, resp.\ $\lambda^{(3s+d-r)/(2s)}\geq\lambda_\delta^{(3s+d-r)/(2s)}$.
Hence, using $\lambda\leq(\sigma^2/(n\rho^2))^{\gamma(s,d-1)}+\lambda_\delta$, the risk is at most $C\rho^2(\Lambda(\sigma^2/(n\rho^2))^{\gamma(s,d-1)}+\lambda_\delta)$.
This is the bound of (a) in \cref{thm:elliptic}, resp.\ \cref{thm:hyperbolic}.

Combining both cases, the estimator has the asserted worst-case risk over $\F(\rho,\delta)$ in each residual regime, which proves \cref{thm:elliptic} and the upper bound of \cref{thm:hyperbolic}.
\end{proof}

\subsection{Proof of the elliptic minimax rate}
\begin{proof}
As in the proof of \cref{thm:universal}, the residual term $\delta^{2d/(2r+d)}(\sigma^2/n)^{\gamma(r,d)}$ of the displayed rate equals the Sobolev term $\rho^2(\sigma^2/(n\rho^2))^{\gamma(s,d)}$ times $(\delta/\delta_n)^{2d/(2r+d)}$.
Hence the minimum in the displayed rate is the residual term when $\delta\leq\delta_n$ and the Sobolev term otherwise, and cases (a) and (b) of \cref{thm:elliptic} give the upper bound.
For the lower bound, after enlarging $C_0$, \cref{prop:ellexact} and $\F(\rho,0)\subseteq\F(\rho,\delta)$ bound $\M_n(\F(\rho,\delta))$ from below by $c$ times the exact-solution term.
Let $C_0'$ be the constant of \cref{thm:universal}, which applies since $s>r>d/2$.
If $n\delta^2\geq C_0'\sigma^2$, that theorem bounds $\M_n(\F(\rho,\delta))$ from below by $c$ times the minimum, after enlarging $C_0$ once more.
Otherwise, including $\delta=0$, the minimum is at most $\delta^{2d/(2r+d)}(\sigma^2/n)^{\gamma(r,d)}=(\sigma^2/n)(n\delta^2/\sigma^2)^{d/(2r+d)}\leq C\sigma^2/n\leq C\rho^2(\sigma^2/(n\rho^2))^{\gamma(s,d-1)}$, by $\sigma^2/(n\rho^2)\leq1$.
In both cases, $\M_n(\F(\rho,\delta))$ is at least $c$ times the maximum, hence $c/2$ times the sum, of the two terms of the displayed rate.
Thus, on domains with a flat face and under the conditions of \cref{thm:elliptic}, the minimax risk lies between constant multiples of the displayed rate.
\end{proof}

\section{Source terms and adaptation}
\label{app:gain}

\subsection{Proof of \texorpdfstring{\cref{prop:source}}{the source-term proposition}}
\begin{proof}
(i) For $h\in\F(\rho/2,\delta)$, linearity gives $\D(f_0+h)-g=\D(h)$ and the triangle inequality gives $\|f_0+h\|_{H^s(\Omega)}\leq\rho$, so $f_0+\F(\rho/2,\delta)\subseteq\F_g(\rho,\delta)$.
Since $f_0$ is known, the map $\hat f\mapsto f_0+\hat f((X_i,Y_i-f_0(X_i))_{i\leq n})$ is a bijection between estimators that turns regression on $\F(\rho/2,\delta)$ into regression on $f_0+\F(\rho/2,\delta)$ with the same noise and the same prediction error.
Hence $\M_n(\F_g(\rho,\delta))\geq\M_n(f_0+\F(\rho/2,\delta))=\M_n(\F(\rho/2,\delta))$.

(ii) This is \cref{lem:upper}.

The lower bounds over $\F_g(\rho,\delta)$ follow from (i).
The upper-bound proofs in \cref{app:upper} use $\F(\rho,\delta)$ only through \cref{lem:upper} with $g=0$; replacing it by \cref{lem:upper} with the given $g$ yields the upper bounds of \cref{thm:elliptic,thm:hyperbolic} over $\F_g(\rho,\delta)$ with the same parameters, since $\mathcal N(\lambda,\mu)$ does not depend on $g$.
This proves (i), (ii) and the stated consequences.
\end{proof}

\subsection{Selection on an independent validation sample}
\begin{lem}[Finite validation selection]
\label[lem]{lem:validation}
Let $|f^\star|\leq R$, and assume $\EE[\exp(t\varepsilon)\mid X]\leq\exp(\sigma^2t^2/2)$ for every $t\in\R$.
Let $u_1,\dots,u_J$, $J\geq2$, be functions with values in $[-R,R]$, independent of a sample $(X_i,Y_i)_{i\leq n_2}$ of independent copies of $(X,Y)$.
Let $\hat\jmath$ minimize $\sum_{i\leq n_2}(Y_i-u_j(X_i))^2$, with ties broken by a fixed order.
There is a universal constant $C$ such that
\[
\EE\big[\|u_{\hat\jmath}-f^\star\|_{L^2(\PP_X)}^2\mid u_1,\dots,u_J\big]
\leq3\min_{j\leq J}\|u_j-f^\star\|_{L^2(\PP_X)}^2+C(R^2+\sigma^2)\frac{\log J}{n_2}.
\]
\end{lem}
\begin{proof}
Condition on the candidates, write $L_j=\|u_j-f^\star\|^2_{L^2(\PP_X)}$, and fix $j_0$ with $L_{j_0}=\min_jL_j$.
For each $j$, let $h=u_j-u_{j_0}$ and $Z=(Y-u_j(X))^2-(Y-u_{j_0}(X))^2=q(X)-2\varepsilon h(X)$, with $q=h(u_j+u_{j_0}-2f^\star)$; then $\EE Z=L_j-L_{j_0}$, because $\EE(\varepsilon\mid X)=0$.
Since $|h|\leq2R$, $|u_j+u_{j_0}-2f^\star|\leq4R$ and $\EE h(X)^2\leq2(L_j+L_{j_0})$, we have $q=(u_j-f^\star)^2-(u_{j_0}-f^\star)^2\in[-4R^2,4R^2]$, hence $|q-\EE q|\leq8R^2$, and $\EE q(X)^2\leq32R^2(L_j+L_{j_0})$.
Integrating the tail bound $\PP(|\varepsilon|>t\mid X)\leq2e^{-t^2/(2\sigma^2)}$ gives $\EE[|\varepsilon|^k\mid X]\leq k!(C\sigma)^k$, hence $\EE|2\varepsilon h(X)|^k\leq k!(C\sigma)^k(2R)^{k-2}\EE h(X)^2$ for $k\geq2$.
With $\sigma R\leq R^2+\sigma^2$, both parts give, for every integer $k\geq2$,
\[
\EE|Z-\EE Z|^k\leq\frac{k!}2\,b^{k-2}v,
\qquad
b=C(R^2+\sigma^2),
\qquad
v=C(R^2+\sigma^2)(L_j+L_{j_0}).
\]
Bernstein's inequality \citep[Theorem~2.10]{boucheron2013concentration}, applied to the centered variables $\EE Z-Z_i$ with variance factor $n_2v$ and scale factor $b$ and then divided by $n_2$, and Young's inequality give, with probability at least $1-e^{-t}$,
\[
\EE Z-\frac1{n_2}\sum_{i\leq n_2}Z_i\leq\sqrt{\frac{2vt}{n_2}}+\frac{bt}{n_2}\leq\frac{L_j+L_{j_0}}2+C(R^2+\sigma^2)\frac t{n_2}.
\]
Take $t=\log J+u$ with $u\geq0$; by a union bound, this holds for all $j$ simultaneously with probability at least $1-e^{-u}$.
The empirical loss difference of $\hat\jmath$ is nonpositive, so on this event $L_{\hat\jmath}-L_{j_0}\leq(L_{\hat\jmath}+L_{j_0})/2+C(R^2+\sigma^2)(\log J+u)/n_2$, that is, $L_{\hat\jmath}\leq3L_{j_0}+C(R^2+\sigma^2)(\log J+u)/n_2$.
Integrating this tail bound over $u\geq0$ and using $\log J\geq\log2$ proves the lemma.
\end{proof}

\subsection{Proof of \texorpdfstring{\cref{thm:adaptation}}{the adaptation theorem}}
\begin{proof}
Take $C_0\geq1$, so that $\sigma^2/(n_1\rho^2)\leq1$; since $s>r$, $\delta_{n_1}=\rho(\sigma^2/(n_1\rho^2))^{(s-r)/(2s+d)}\leq\rho$.
For $n\geq4$, $n/3\leq n_1\leq n/2\leq n_2$.

\emph{Selection.}
The constant $C_{\mathrm{emb}}$ exists: by \cref{cor:extension}, every $f\in H^s(\Omega)$ agrees on $\Omega$ with a periodic $g$, and $|g|\leq\sum_k|\hat g_k|\leq Ca_s^{1/2}\|g\|_{H^s(\T)}\leq C\|f\|_{H^s(\Omega)}$.
Every $f\in\Hc(\rho)$ therefore satisfies $|f|\leq C_{\mathrm{emb}}\rho$, and clipping is the pointwise projection onto an interval containing the values of $f^\star$, so $\|\check f_{\lambda,\mu}-f^\star\|_{L^2(\PP_X)}\leq\|\hat f_{\lambda,\mu}-f^\star\|_{L^2(\PP_X)}$, where $\hat f_{\lambda,\mu}$ is the unclipped fit.
The grid $\mathcal G$ has at most $2+\log_2n_1$ points, so the number $J\geq2$ of candidates satisfies $\log J\leq C\log\log(3+n)$.
\Cref{lem:validation} with $R=C_{\mathrm{emb}}\rho$, followed by expectation over the training sample and $\EE\min_j\leq\min_j\EE$, gives
\[
\EE\|\tilde f_n-f^\star\|^2_{L^2(\PP_X)}\leq3\min_{\lambda\in\mathcal G,\,\mu\in\mathcal G\cup\{0\}}\EE\|\hat f_{\lambda,\mu}-f^\star\|^2_{L^2(\PP_X)}+C(\rho^2+\sigma^2)\frac{\log\log(3+n)}n.
\]

\emph{A good candidate.}
Fix $\delta\geq0$ and $f^\star\in\F(\rho,\delta)$.
Since $s>r$, $\delta_n\leq\delta_{n_1}$, and $\ell$ at sample size $n_1$ is at most $\ell$ at sample size $n$.
Hence, if $\delta$ is in regime (a) of \cref{thm:elliptic}, resp.\ \cref{thm:hyperbolic}, at sample size $n$, it is also in regime (a) at sample size $n_1$; let $(\lambda_\star,\mu_\star)$ be the parameters of (a) at sample size $n_1$.
Otherwise, let $(\lambda_\star,\mu_\star)$ be the parameters of (b) at sample size $n_1$.
The assumptions of the theorem are the sample-size assumptions of \cref{thm:elliptic,thm:hyperbolic} at $n_1$, so $\lambda_\star\in[1/n_1,1]$, as shown in \cref{app:upper}.
If $\mu_\star>0$, then $\delta\leq\delta_{n_1}\leq\rho$ (in the hyperbolic case, $\ell\geq\log3>1$ because $\sigma^2/(n_1\rho^2)\leq1$), so $\lambda_\star\rho^2/\delta^2\geq\lambda_\star$; the caps $\lambda_\star^{r(d-1)/(sd)}$ and $\lambda_\star^{(r-1)/s}$ have exponents in $[0,1]$, so $\mu_\star\in[\lambda_\star,1]$.
The smallest point of $\mathcal G$ above any $t\in[1/n_1,1]$ lies in $[t,\min(2t,1)]$.
Let $\lambda'$ be this point for $t=\lambda_\star$, let $\mu'=0$ if $\mu_\star=0$, and let $\mu'$ be this point for $t=\mu_\star$ otherwise.
The truncation $m_{\lambda',\mu'}$ satisfies the requirement of \cref{lem:upper}, and $(\lambda',\mu')\in(0,1]\times[0,1]$.
Since the sub-Gaussian bound implies (A2), \cref{lem:upper} with $n_1$ observations bounds the risk of $\hat f_{\lambda',\mu'}$ by $C(\lambda'\rho^2+\mu'\delta^2+\sigma^2\mathcal N(\lambda',\mu')/n_1)$.
Here $\lambda'\rho^2+\mu'\delta^2\leq2(\lambda_\star\rho^2+\mu_\star\delta^2)$, and the bounds of \cref{thm:count} are nonincreasing in $\lambda$ and $\mu$ and uniform in $m$, so they bound $\mathcal N(\lambda',\mu')$ by their values at $(\lambda_\star,\mu_\star)$.
The computations of \cref{app:upper} then bound this risk by $2C$ times the bound of (a), resp.\ (b), at sample size $n_1$.
Since $n_1\geq n/3$, every power of $1/n_1$ appearing there, with exponent at most one, is at most $3$ times the same power of $1/n$, and the logarithms at $n_1$ are at most their values at $n$.
So $\EE\|\hat f_{\lambda',\mu'}-f^\star\|^2_{L^2(\PP_X)}$ is at most a constant times the bound of \cref{thm:elliptic}, resp.\ \cref{thm:hyperbolic}, at sample size $n$, in the regime of $\delta$.

Inserting this candidate in the selection bound shows that $\tilde f_n$, whose construction does not use $\delta$, has worst-case risk over $\F(\rho,\delta)$ at most a constant times the upper bound of \cref{thm:elliptic}, resp.\ \cref{thm:hyperbolic}, plus $C(\rho^2+\sigma^2)\log\log(3+n)/n$, for every $\delta\geq0$.
\end{proof}

\section{The fast algorithm}
\label{app:algo}

\subsection{Proof of \texorpdfstring{\cref{thm:algo}}{the complexity theorem}}
\begin{proof}
All operation counts start from the exact generating sequences of $\hat\Sigma$ and $G$ and the exact vector $v_n$.

\emph{Cost of one iteration.}
The entries of $\hat\Sigma$ and $G$ depend only on $l-k\in[-2m,2m]^d$, so each matrix is $d$-level Toeplitz and embeds into a $d$-level circulant matrix of size $(4m+1)^d$, which the discrete Fourier transform diagonalizes (\citealp[Theorem~4.8.2 and Problem~P4.8.4]{golub2013matrix} in one dimension; the $d$-level case is the tensor product).
Each Toeplitz product therefore costs $O(m^d\log(2+m))$ operations.
Since $S$ and $D$ are diagonal, so do a product by $A_n$, the application of the preconditioner $\lambda S^2$, and one conjugate-gradient iteration.

From now on, take the parameters of \cref{thm:elliptic} or of the upper bound of \cref{thm:hyperbolic}; then $s>d/2$ and $1/n\leq\lambda\leq1$, as shown in \cref{app:upper}, and $0\leq\mu\leq1$, because $\mu$ is at most a nonnegative power of $\lambda$.

\emph{Conditioning.}
The preconditioned matrix $Q=(\lambda S^2)^{-1/2}A_n(\lambda S^2)^{-1/2}$ equals $I+\lambda^{-1}S^{-1}\hat\Sigma S^{-1}+\mu\lambda^{-1}S^{-1}D^\ast GDS^{-1}\succeq I$.
For every $\theta$, the Cauchy--Schwarz inequality and the feature bound of \cref{app:tools} give $\theta^\ast S^{-1}\hat\Sigma S^{-1}\theta\leq\sup_x|\phi(x)^\ast S^{-1}\theta|^2\leq a_s\|\theta\|_2^2$.
Parseval's identity gives $\theta^\ast G\theta=\|f_\theta\|^2_{L^2(\Omega)}\leq(4L)^d\|\theta\|_2^2$, and $\sup_{k\in\Z^d}|P(i\pi k/(2L))|^2/(1+\|k\|_2^{2s})<\infty$ because $r\leq s$, so $\|S^{-1}D^\ast GDS^{-1}\|_{\mathrm{op}}\leq C$ independently of $m$.
Since $\lambda,\mu\leq1$, it follows that $I\preceq Q\preceq(C/\lambda)I$.

\emph{Accuracy.}
Write $\|z\|_A^2=z^\ast Az$ for $A\succ0$.
Starting from $\theta^{(0)}=0$, the convergence bound of preconditioned conjugate gradients with condition number at most $C/\lambda$ gives
\[
\|\theta^{(t)}-\hat\theta\|_{A_n}^2\leq4\exp(-ct\sqrt\lambda)\|\hat\theta\|_{A_n}^2
\]
\citep[Section~11.5 and Eq.~(11.3.27)]{golub2013matrix}.
If the residual vanishes at some iteration $t\leq K$, then $\theta^{(t)}=\hat\theta$ because $A_n\succ0$, and the bounds below hold trivially; otherwise the estimate applies with $t=K$.
Let $\Phi^\ast=(\phi(X_1),\dots,\phi(X_n))$ and $Y=(Y_1,\ldots,Y_n)^\top$, so that $v_n=\Phi^\ast Y/n$.
Since $A_n\succeq\Phi^\ast\Phi/n$, we have $\Phi A_n^{-1}\Phi^\ast/n\preceq I_n$, hence $\|\hat\theta\|_{A_n}^2=v_n^\ast A_n^{-1}v_n\leq\|Y\|_2^2/n$ and $\EE\|\hat\theta\|_{A_n}^2\leq\kappa\rho^2+\sigma^2$ by (A1)--(A2).
Moreover, $\Sigma\preceq\kappa(4L)^dI$ by Parseval's identity and $A_n\succeq\lambda S^2\succeq\lambda I$, so
\[
\EE\|f_{\theta^{(t)}}-f_{\hat\theta}\|^2_{L^2(\PP_X)}\leq\kappa(4L)^d\lambda^{-1}\EE\|\theta^{(t)}-\hat\theta\|_{A_n}^2\leq Cn\exp(-ct\sqrt\lambda)(\kappa\rho^2+\sigma^2),
\]
using $\lambda\geq1/n$.
For fixed $\rho$ and $\sigma$, a large enough $C_K$ makes the right-hand side at most $\sigma^2/n$ whenever $t\geq C_K\lambda^{-1/2}\log(2+n)$.
Then $\EE\|f_{\theta^{(t)}}-f^\star\|^2_{L^2(\PP_X)}\leq2\EE\|f_{\hat\theta}-f^\star\|^2_{L^2(\PP_X)}+2\sigma^2/n$, and $\sigma^2/n\leq\rho^2(\sigma^2/(n\rho^2))^{\gamma(s,d-1)}$ because $\sigma^2/(n\rho^2)\leq1$.
Since each of the bounds of \cref{thm:elliptic,thm:hyperbolic} is at least $C\rho^2(\sigma^2/(n\rho^2))^{\gamma(s,d-1)}$, because $\log(2+n\rho^2/\sigma^2)\geq1$ and $\gamma(s,d)<\gamma(s,d-1)$, this gives these upper bounds with $C$ replaced by $2C+2$.

\emph{Total cost.}
Let $K=\lceil C_K\lambda^{-1/2}\log(2+n)\rceil$ and $m=\lceil\max\{\lambda^{-1/(2s)},(\mu/\lambda)^{1/(2s-2r)}\}\rceil$.
In the statistical theorems, either $\mu=0$, or $\mu\leq\lambda^a$ with $a=(r-1)/s$ for hyperbolic and $a=r(d-1)/(sd)$ for elliptic operators.
Hence $m\leq2\lambda^{-b/(2s)}$, with $b=1$ if $\mu=0$ and $b=(1-a)s/(s-r)\geq1$ otherwise, that is,
\[
b=1+\frac1{s-r}\ \text{(hyperbolic)},
\qquad
b=1+\frac r{d(s-r)}\ \text{(elliptic)}.
\]
With $\lambda\geq1/n$, this gives $\log(2+m)\leq C\log(2+n)$, so the $K$ iterations cost at most $C\lambda^{-1/2}m^d\log^2(2+n)$ operations.
Since $\lambda\geq(\sigma^2/(n\rho^2))^{\gamma(s,d-1)}$,
\[
\lambda^{-1/2}m^d\leq C\lambda^{-(s+bd)/(2s)}\leq C\Big(\frac{n\rho^2}{\sigma^2}\Big)^{(s+bd)/(2s+d-1)}.
\]
The exponent is at most one if and only if $(b-1)d\leq s-1$.
This holds for $b=1$ since $s>1$, for hyperbolic operators if and only if $(s-1)(s-r)\geq d$, and for elliptic operators if and only if $(s-1)(s-r)\geq r$, that is, $s(s-r-1)\geq0$, that is, $s\geq r+1$.
Under these conditions, $\lambda^{-1/2}m^d\leq Cn\rho^2/\sigma^2$, so for fixed $\rho$ and $\sigma$ the iterations cost $O(n\log^2(2+n))$ operations in total.
This proves the per-iteration cost, the risk bound after $K\geq C_K\lambda^{-1/2}\log(2+n)$ iterations, and the total cost stated in \cref{thm:algo}.
\end{proof}

\subsection{Experiments}
\label{app:exp}
\begin{figure}[t]
\centering
\includegraphics[width=0.32\textwidth]{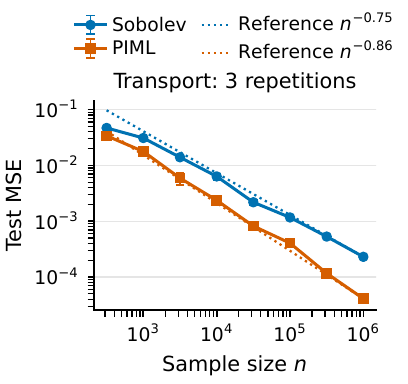}
\includegraphics[width=0.32\textwidth]{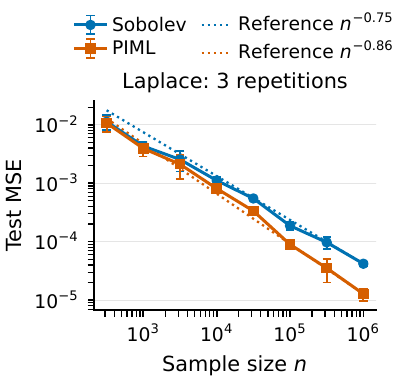}
\includegraphics[width=0.32\textwidth]{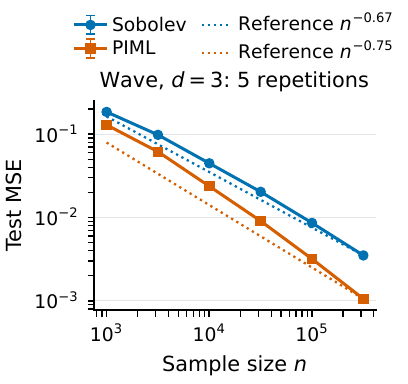}
\caption{Mean test errors for exact physics, with $\pm2$ estimated standard errors.
Left and middle: transport and Laplace operators in dimension two, three repetitions, reference exponents $0.75$ and $6/7$.
Right: wave operator in dimension three, five repetitions, reference exponents $2/3$ and $3/4$.}
\label{fig:ops}

\end{figure}
\paragraph{Common setting.}
All experiments use uniform design on $\Omega=(-1/2,1/2)^d$, Gaussian noise with $\sigma=0.5$, $s=3$ and $L=1$, and test errors on 20,000 independent uniform points.
All targets are finite trigonometric or exponential sums, hence smoother than the assumed Sobolev class.
The schedules use the powers of $n$ of \cref{thm:elliptic,thm:hyperbolic} with multiplicative constants set to one, rather than their radius and noise calibration; for the two-dimensional wave target, the sample-size condition $n\geq(\rho/\sigma)^{12}$ is about $4.5\cdot10^{28}$, so these illustrations do not verify the finite-sample guarantees.
The code divides $G$ by $(4L)^d$, so its weight $\mu_{\mathrm{code}}$ corresponds to $\mu=\mu_{\mathrm{code}}/(4L)^d$.

\paragraph{Two-dimensional wave operator.}
For $\D=\partial_t^2-\partial_x^2$, the target is $u_1(x+t)+u_2(x-t)$, with $u_1(z)=\sum_{k=1}^{300}k^{-3.6}\cos(\pi kz+\varphi_k)$ and $u_2(z)=\sum_{k=1}^{300}k^{-3.6}\sin(\pi kz+\varphi'_k)$, where the phases are independent uniform draws fixed across repetitions; its $H^3$ norm is approximately $122$.
The Sobolev schedule is $\lambda=n^{-3/4}$, $\mu=0$, $m=\lceil2n^{1/8}\rceil$, and the wave schedule is $\lambda=n^{-6/7}$, $\mu_{\mathrm{code}}=n^{-2/7}$, $m=\lceil n^{2/7}\rceil$, so the reference exponents are $\gamma(3,2)=0.75$ and $\gamma(3,1)=6/7$.
Ten repetitions are saved for each $n\in\{10^2,10^{2.5},\ldots,10^7\}$ (half-decade steps, $n$ rounded to the nearest integer).

\paragraph{Imperfect physics.}
For $\delta>0$, the target is $u_1(x+t)+u_2(x-t)+\delta u_3/\|\D(u_3)\|_{L^2(\Omega)}$, so that its residual is $\delta$, where $u_3$ is a Fourier sum over the half-disk $1\leq\|k\|_2\leq40$ with standard Gaussian amplitudes multiplied by $\|k\|_2^{-4.2}$ and uniform phases; here $\|u_3\|_{H^3(\Omega)}/\|\D(u_3)\|_{L^2(\Omega)}\approx7.1$.
The residuals $\delta\approx0.979$ and $\delta\approx11.39$ give $\delta/\delta_n=0.1$ and $1$ at $n=10^4$, where $\delta_n$ uses the perturbed target's $H^3$ norm, computed by quadrature (its square is quadratic in $\delta$).
The Sobolev grid is $\lambda\in\{10^{-7},10^{-6.5},\ldots,10^{-1}\}$, and the physics-informed grid is $\lambda\in\{10^{-7},10^{-6},\ldots,10^{-1}\}$, $\mu_{\mathrm{code}}\in\{0,10^{-4},10^{-3},\ldots,10^4\}$; the truncations are $\lceil2n^{1/8}\rceil$ and $\lceil1.5n^{2/7}\rceil$.
With ten repetitions for each $n\in\{10^2,10^{2.5},\ldots,10^7\}$, and the same samples for the three residuals, the figure reports the ratio of the mean minimal test errors; the same evaluation points serve for selection and reporting.
Up to $n=3\cdot10^4$, the selected $\lambda$ lies inside the grid for both estimators in every repetition; at $n=10^5$, the smallest value $\lambda=10^{-7}$ is selected in at most three repetitions.
At $n=3\cdot10^5$, it is selected in up to nine repetitions, depending on the estimator and the residual, and for $n\geq10^6$ in almost all repetitions: the grid, which is not calibrated to $\sigma^2/(n\rho^2)$, then limits the comparison, so the points with $n\geq10^6$ should not be interpreted.
At every sample size, the selected physics-informed fit stops at the cap of $5000$ conjugate-gradient iterations in some repetitions, so the oracle may select an early-stopped iterate; for $n\geq10^6$, this happens in almost all repetitions.
The full grids of test errors and iteration counts are saved with the code.

\paragraph{Three-dimensional wave operator.}
For $\D=\partial_t^2-\Delta_{x'}$, the target is $f^\star(t,x')=\sum_{k\in\Z^2,\,1\leq\|k\|_2\leq20}\|k\|_2^{-4.5}Z_k\cos(\pi\|k\|_2t+\pi\langle k,x'\rangle+\varphi_k)$, with fixed standard Gaussian $Z_k$ and uniform phases $\varphi_k$.
The Sobolev schedule is $\lambda=n^{-2/3}$, $m=\lceil2n^{1/9}\rceil$, and the wave schedule is $\lambda=n^{-3/4}$, $\mu_{\mathrm{code}}=n^{-1/4}$, $m=\lceil n^{1/4}\rceil$.
Five repetitions for $n$ from $10^3$ to $3\cdot10^5$ give error ratios $0.70$, $0.63$, $0.53$, $0.45$, $0.37$ and $0.30$ (\cref{fig:ops}, right); at the last size, the fits take 138 to 143 iterations and 69.8 to 74.2 seconds.

\paragraph{Transport and Laplace operators.}
In dimension two, the transport operator is $\partial_t+\partial_x$, with exact target $u_1(x-t)$, and the Laplace target is $f^\star(t,x)=\sum_{k=1}^{300}k^{-3.1}e^{\pi k(t-1/2)}\cos(\pi kx+\varphi_k)$.
Both physics-informed schedules use $\lambda=n^{-6/7}$ and $m=\lceil n^{2/7}\rceil$, with $\mu_{\mathrm{code}}=1$ for transport and $\mu_{\mathrm{code}}=\lambda^{1/3}$ for Laplace, and the Sobolev schedule is as for the wave operator.
Three repetitions are saved, and the slopes are $(0.683,0.840)$ and $(0.689,0.838)$, respectively (\cref{fig:ops}, left and middle).

\paragraph{Solver and timing provenance.}
The sequences of $\hat\Sigma$ and $v_n$ are computed by \texttt{FINUFFT} with tolerance $10^{-12}$ in dimension two and $10^{-10}$ in dimension three, and Toeplitz products use zero-padded FFTs.
For \cref{fig:exp}, the fits follow \cref{alg:fast}: conjugate gradients preconditioned by $\lambda S^2$ and started at $\theta^{(0)}=0$, with relative residual tolerance $10^{-8}$ for the rates and $10^{-6}$ for the timing.
The oracle grid search instead warm-starts along decreasing values of $\lambda$, which does not change the minimizer, uses the diagonal of the system as preconditioner, a tolerance of $10^{-7}$ and at most $5000$ iterations.
The timing uses the two-dimensional wave target and schedule, with ten fits for each $n\in\{10^2,10^{2.5},\ldots,10^7\}$ on one thread of an Apple M2 Pro, and includes the computation of the Toeplitz sequences and the solver, but not data generation and test prediction.
The archived runs of \cref{fig:ops} were initialized at $(\lambda S^2)^{-1}v_n$, and their run files do not identify every preconditioner.

\label{endapp}
\bibliography{biblio13}

\begin{thebibliography}{39}
\providecommand{\natexlab}[1]{#1}
\providecommand{\url}[1]{\texttt{#1}}
\expandafter\ifx\csname urlstyle\endcsname\relax
  \providecommand{\doi}[1]{doi: #1}\else
  \providecommand{\doi}{doi: \begingroup \urlstyle{rm}\Url}\fi

\bibitem[Arnone et~al.(2022)Arnone, Kneip, Nobile, and Sangalli]{arnone2022consistency}
Eleonora Arnone, Alois Kneip, Fabio Nobile, and Laura~M. Sangalli.
\newblock Some first results on the consistency of spatial regression with partial differential equation regularization.
\newblock \emph{Statistica Sinica}, 32:\penalty0 209--238, 2022.

\bibitem[Bach(2024)]{bach2024learning}
Francis Bach.
\newblock \emph{Learning Theory from First Principles}.
\newblock MIT Press, Cambridge, Massachusetts, 2024.

\bibitem[Barnett et~al.(2019)Barnett, Magland, and af~Klinteberg]{barnett2019parallel}
Alexander~H. Barnett, Jeremy Magland, and Ludvig af~Klinteberg.
\newblock A parallel nonuniform fast {F}ourier transform library based on an ``exponential of semicircle'' kernel.
\newblock \emph{SIAM Journal on Scientific Computing}, 41\penalty0 (5):\penalty0 C479--C504, 2019.

\bibitem[Bona-Pellissier et~al.(2026)Bona-Pellissier, Meanti, Santacesaria, and Rosasco]{bonapellissier2026piks}
Joachim Bona-Pellissier, Giacomo Meanti, Matteo Santacesaria, and Lorenzo Rosasco.
\newblock {PIKS}: {U}niversal physics-informed kernel methods.
\newblock \emph{arXiv:2607.27062}, 2026.

\bibitem[Boucheron et~al.(2013)Boucheron, Lugosi, and Massart]{boucheron2013concentration}
St{\'e}phane Boucheron, G{\'a}bor Lugosi, and Pascal Massart.
\newblock \emph{Concentration Inequalities: A Nonasymptotic Theory of Independence}.
\newblock Oxford University Press, Oxford, 2013.

\bibitem[Brogat-Motte et~al.(2026)Brogat-Motte, Bona-Pellissier, Meanti, and Rosasco]{brogatmotte2026fast}
Luc Brogat-Motte, Joachim Bona-Pellissier, Giacomo Meanti, and Lorenzo Rosasco.
\newblock Fast learning rates for physics-informed kernel methods.
\newblock \emph{arXiv:2609.18901}, 2026.

\bibitem[Caponnetto and {De Vito}(2007)]{caponnetto2007optimal}
A.~Caponnetto and E.~{De Vito}.
\newblock Optimal rates for the regularized least-squares algorithm.
\newblock \emph{Foundations of Computational Mathematics}, 7:\penalty0 331–368, 2007.

\bibitem[Dai(2025)]{dai2025nonparametric}
Xiaowu Dai.
\newblock Nonparametric estimation via partial derivatives.
\newblock \emph{Journal of the Royal Statistical Society: Series B}, 87:\penalty0 319--336, 2025.

\bibitem[{De Ryck} and Mishra(2024)]{deryck2024numerical}
Tim {De Ryck} and Siddhartha Mishra.
\newblock Numerical analysis of physics-informed neural networks and related models in physics-informed machine learning.
\newblock \emph{Acta Numerica}, 33:\penalty0 633--713, 2024.

\bibitem[Di~Nezza et~al.(2012)Di~Nezza, Palatucci, and Valdinoci]{dinezza2012hitchhiker}
Eleonora Di~Nezza, Giampiero Palatucci, and Enrico Valdinoci.
\newblock Hitchhiker's guide to the fractional {S}obolev spaces.
\newblock \emph{Bulletin des Sciences Math{\'e}matiques}, 136\penalty0 (5):\penalty0 521--573, 2012.

\bibitem[Doum{\`e}che et~al.(2024)Doum{\`e}che, Bach, Biau, and Boyer]{doumeche2024physicsinformed}
Nathan Doum{\`e}che, Francis Bach, G{\'e}rard Biau, and Claire Boyer.
\newblock Physics-informed machine learning as a kernel method.
\newblock In Shipra Agrawal and Aaron Roth, editors, \emph{Proceedings of Thirty Seventh Conference on Learning Theory}, volume 247 of \emph{Proceedings of Machine Learning Research}, pages 1399--1450. PMLR, 2024.

\bibitem[Doum{\`e}che et~al.(2025)Doum{\`e}che, Biau, and Boyer]{doumeche2025convergence}
Nathan Doum{\`e}che, G{\'e}rard Biau, and Claire Boyer.
\newblock On the convergence of {PINN}s.
\newblock \emph{Bernoulli}, 31:\penalty0 2127--2151, 2025.

\bibitem[Doum{\`e}che et~al.(2026)Doum{\`e}che, Bach, Biau, and Boyer]{doumeche2026fast}
Nathan Doum{\`e}che, Francis Bach, G{\'e}rard Biau, and Claire Boyer.
\newblock Fast kernel methods: Sobolev, physics-informed, and additive models.
\newblock In \emph{Proceedings of the 43rd International Conference on Machine Learning}, volume 306 of \emph{Proceedings of Machine Learning Research}, pages 26319--26340. PMLR, 2026.

\bibitem[Doumèche et~al.(2025)Doumèche, Bach, Biau, and Boyer]{doumeche2025pikl}
Nathan Doumèche, Francis Bach, Gérard Biau, and Claire Boyer.
\newblock Physics-informed kernel learning.
\newblock \emph{Journal of Machine Learning Research}, 26\penalty0 (124):\penalty0 1--39, 2025.

\bibitem[Dupont and Scott(1980)]{dupont1980polynomial}
Todd Dupont and Ridgway Scott.
\newblock Polynomial approximation of functions in {S}obolev spaces.
\newblock \emph{Mathematics of Computation}, 34\penalty0 (150):\penalty0 441--463, 1980.
\newblock \doi{10.1090/S0025-5718-1980-0559195-7}.

\bibitem[Evans(1998)]{evans1998partial}
L.~C. Evans.
\newblock \emph{Partial Differential Equations}, volume~19 of \emph{Graduate Studies in Mathematics}.
\newblock American Mathematical Society, Providence, RI, 1st edition, 1998.

\bibitem[Golub and {Van Loan}(2013)]{golub2013matrix}
Gene~H. Golub and Charles~F. {Van Loan}.
\newblock \emph{Matrix Computations}.
\newblock The Johns Hopkins University Press, Baltimore, 4th edition, 2013.

\bibitem[Grisvard(2011)]{grisvard1985elliptic}
Pierre Grisvard.
\newblock \emph{Elliptic Problems in Nonsmooth Domains}, volume~69 of \emph{Classics in Applied Mathematics}.
\newblock SIAM, Philadelphia, 2011.
\newblock Reprint of the 1985 original (Pitman, Boston).

\bibitem[Gy{\"o}rfi et~al.(2002)Gy{\"o}rfi, Kohler, Krzy{\.z}ak, and Walk]{gyorfi2002distribution}
L{\'a}szl{\'o} Gy{\"o}rfi, Michael Kohler, Adam Krzy{\.z}ak, and Harro Walk.
\newblock \emph{A Distribution-Free Theory of Nonparametric Regression}.
\newblock Springer Series in Statistics. Springer, New York, 2002.

\bibitem[Karniadakis et~al.(2021)Karniadakis, Kevrekidis, Lu, Wang, Perdikaris, and Yang]{karniadakis2021piml}
George~Em Karniadakis, Ioannis~G. Kevrekidis, Lu~Lu, Sifan Wang, Paris Perdikaris, and Liu Yang.
\newblock Physics-informed machine learning.
\newblock \emph{Nature Reviews Physics}, 3:\penalty0 422--440, 2021.

\bibitem[Kennedy and O'Hagan(2001)]{kennedy2001bayesian}
Marc~C. Kennedy and Anthony O'Hagan.
\newblock Bayesian calibration of computer models.
\newblock \emph{Journal of the Royal Statistical Society: Series B}, 63:\penalty0 425--464, 2001.

\bibitem[Konen and Nickl(2025)]{konen2025data}
Dimitri Konen and Richard Nickl.
\newblock Data assimilation with the {2D} {N}avier--{S}tokes equations: {O}ptimal {G}aussian asymptotics for the posterior measure.
\newblock \emph{The Annals of Statistics}, 2025.
\newblock To appear.

\bibitem[Koshizuka and Sato(2025)]{koshizuka2025understanding}
Takeshi Koshizuka and Issei Sato.
\newblock Understanding generalization in physics informed models through affine variety dimensions.
\newblock In \emph{Advances in Neural Information Processing Systems}, volume~38, 2025.

\bibitem[Li and Yao(2024)]{li2024physicsencoded}
Tongyu Li and Fang Yao.
\newblock Physics-encoded spatio-temporal regression.
\newblock \emph{arXiv:2410.04170}, 2024.

\bibitem[Lu et~al.(2022)Lu, Chen, Lu, Ying, and Blanchet]{lu2022machine}
Yiping Lu, Haoxuan Chen, Jianfeng Lu, Lexing Ying, and Jose Blanchet.
\newblock Machine learning for elliptic {PDE}s: {F}ast rate generalization bound, neural scaling law and minimax optimality.
\newblock In \emph{International Conference on Learning Representations}, 2022.

\bibitem[Marcondes(2025)]{marcondes2025complexity}
Diego Marcondes.
\newblock Complexity dependent error rates for physics-informed statistical learning via the small-ball method.
\newblock \emph{arXiv:2510.23149}, 2025.

\bibitem[Mishra and Molinaro(2023)]{mishra2022generalization}
Siddhartha Mishra and Roberto Molinaro.
\newblock Estimates on the generalization error of physics-informed neural networks for approximating {PDE}s.
\newblock \emph{IMA Journal of Numerical Analysis}, 43:\penalty0 1--43, 2023.

\bibitem[Mourtada and Rosasco(2022)]{mourtada2022an}
Jaouad Mourtada and Lorenzo Rosasco.
\newblock An elementary analysis of ridge regression with random design.
\newblock \emph{Comptes Rendus Math{\'e}matique}, 360:\penalty0 1055--1063, 2022.

\bibitem[Nickl et~al.(2020)Nickl, van~de Geer, and Wang]{nickl2019convergence}
Richard Nickl, Sara van~de Geer, and Sven Wang.
\newblock Convergence rates for penalised least squares estimators in {PDE}-constrained regression problems.
\newblock \emph{SIAM/ASA Journal on Uncertainty Quantification}, 8:\penalty0 374--413, 2020.

\bibitem[Raissi et~al.(2019)Raissi, Perdikaris, and Karniadakis]{raissi2019PINN}
M.~Raissi, P.~Perdikaris, and G.~E. Karniadakis.
\newblock Physics-informed neural networks: {A} deep learning framework for solving forward and inverse problems involving nonlinear partial differential equations.
\newblock \emph{Journal of Computational Physics}, 378:\penalty0 686--707, 2019.

\bibitem[Sangalli(2021)]{sangalli2021spatial}
Laura~M. Sangalli.
\newblock Spatial regression with partial differential equation regularisation.
\newblock \emph{International Statistical Review}, 89:\penalty0 505--531, 2021.

\bibitem[Scampicchio et~al.(2026)Scampicchio, Toso, Rickenbach, Anderson, and Zeilinger]{scampicchio2026physics}
Anna Scampicchio, Leonardo~F. Toso, Rahel Rickenbach, James Anderson, and Melanie~N. Zeilinger.
\newblock Physics-informed learning under mixing: {H}ow physical knowledge speeds up learning.
\newblock In \emph{International Conference on Learning Representations}, 2026.

\bibitem[Shin et~al.(2020)Shin, Darbon, and Karniadakis]{shin2020convergence}
Yeonjong Shin, J{\'e}r{\^o}me Darbon, and George~Em Karniadakis.
\newblock On the convergence of physics informed neural networks for linear second-order elliptic and parabolic type {PDE}s.
\newblock \emph{Communications in Computational Physics}, 28:\penalty0 2042--2074, 2020.

\bibitem[Stein(1970)]{stein1970lipschitz}
E.~M. Stein.
\newblock \emph{Singular Integrals and Differentiability Properties of Functions}, volume~30 of \emph{Princeton Mathematical Series}.
\newblock Princeton University Press, Princeton, 1970.

\bibitem[Tahmasebi and Jegelka(2023)]{tahmasebi2023exact}
Behrooz Tahmasebi and Stefanie Jegelka.
\newblock The exact sample complexity gain from invariances for kernel regression.
\newblock In \emph{Advances in Neural Information Processing Systems}, volume~36, pages 55616--55646, 2023.

\bibitem[Tsybakov(2009)]{tsybakov2009introduction}
A.~B. Tsybakov.
\newblock \emph{Introduction to Nonparametric Estimation}.
\newblock Springer, New York, 2009.

\bibitem[Tuo and Wu(2015)]{tuo2015efficient}
Rui Tuo and C.~F.~Jeff Wu.
\newblock Efficient calibration for imperfect computer models.
\newblock \emph{The Annals of Statistics}, 43:\penalty0 2331--2352, 2015.

\bibitem[Zhang and Wang(2025)]{zhang2025sobolev}
Qingwen Zhang and Wenjia Wang.
\newblock Sobolev calibration of imperfect computer models.
\newblock \emph{Journal of the American Statistical Association}, 120:\penalty0 408--418, 2025.

\bibitem[Zhao and Lu(2026)]{zhao2026posterior}
Yuxuan Zhao and Yulong Lu.
\newblock Posterior concentration of {B}ayesian physics-informed neural networks for elliptic {PDE}s.
\newblock In \emph{Proceedings of the 43rd International Conference on Machine Learning}, volume 306 of \emph{Proceedings of Machine Learning Research}. PMLR, 2026.

\end{thebibliography}

\end{document}